\documentclass[11pt]{article}
\usepackage{amssymb,amsmath,amsthm,amscd,amsfonts }
\usepackage{graphicx}
\usepackage{pdfsync}
\usepackage{color}
\usepackage{subfigure}
\usepackage{enumerate}

\long\def\red#1{\textcolor {red}{#1}}

  {\noindent{\bf\ref{#2}#1.}}{\vspace{\baselineskip}}

\makeatletter
\def\ulabel#1#2{\@bsphack\if@filesw {\let\thepage\relax \def\protect{\noexpand\noexpand\noexpand}%
\xdef\@gtempa{\write\@auxout{\string
\newlabel{#1}{{#2 \@currentlabel}{\thepage}}}}}\@gtempa
\if@nobreak \ifvmode\nobreak\fi\fi\fi\@esphack} \makeatother

\newcommand{\compcent}[1]{\vcenter{\hbox{$#1\circ$}}}

\newcommand{\comp}{\mathbin{\mathchoice
{\compcent\scriptstyle}{\compcent\scriptstyle}
{\compcent\scriptscriptstyle}{\compcent\scriptscriptstyle}}}

\newcommand {\N}{\mathbb{N}} 
\newcommand {\Z}{\mathbb{Z}} 
\newcommand {\R}{\mathbb{R}} 

\theoremstyle{plain}
\newtheorem{thm}{Theorem}[section]
\newtheorem{cor}[thm]{Corollary}
\newtheorem{lem}[thm]{Lemma}
\newtheorem{prop}[thm]{Proposition}
\newtheorem{clm}{Claim}
\newtheorem{sclm}{Subclaim}[clm]
\newtheorem*{clm*}{Claim}
\newtheorem*{rmk*}{Remark}
\newtheorem*{thm*}{Theorem}
\newtheorem*{lem*}{Lemma}
\newtheorem*{prop*}{Proposition}

\theoremstyle{definition}
\newtheorem{defn}[thm]{Definition}
\newtheorem{rmk}[thm]{Remark}

\def\co{\colon\thinspace}
\renewcommand{\t}{\mathcal T}
\renewcommand{\top}[1]{\operatorname{\textbf{top}}(#1)}
\renewcommand{\bot}[1]{\operatorname{\textbf{bot}}(#1)}
\newcommand{\topc}[1]{\operatorname{\textbf{top}}_c(#1)}
\newcommand{\botc}[1]{\operatorname{\textbf{bot}}_c(#1)}

\newcommand{\mesh}{\operatorname{mesh}}
\newcommand{\rmesh}{\operatorname{rmesh}}
\newcommand{\pres}[2]{\langle\hspace{.5 mm} #1\hspace{.5 mm} | \hspace{.5 mm} #2\hspace{.5 mm} \rangle}
\newcommand{\con}{\operatorname{Con}^\omega\bigl(X,e,d\bigr)}

\newcommand{\ccon}[3]{\operatorname{Con}^\omega\bigl(#1,(#2),(#3)\bigr)}
\newcommand{\gcon}{\operatorname{Con}^\omega\bigl(G,d\bigr)}
\newcommand{\hcon}{\operatorname{Con}^\omega\bigl(H,d\bigr)}
\newcommand{\as}{$\omega$-almost surely}
\newcommand{\diam}[1]{\operatorname{diam}(#1)}

\renewcommand{\phi}{\varphi}
\newcommand{\im}[1]{\operatorname{im}(#1)}

\newcommand{\lab}[1]{\operatorname{\textbf{Lab}}{(#1)}}
\newcommand{\dv}{{\mathrm{div}}}
\newcommand{\Ball}{{\mathrm{B}}}
\renewcommand{\d}{\operatorname{dist}}

\begin{document}
\title{Asymptotic cones of HNN extensions and amalgamated products}
\author{Curt Kent}
\maketitle

\begin{abstract}
Gromov asked whether an asymptotic cone of a finitely generated group was always simply connected or had uncountable fundamental group.  We prove that Gromov's dichotomy holds for asympotic cones with cut points, as well as, HNN extensions and amalgamated products where the associated subgroups are nicely embedded.  We also show a slightly weaker dichotomy for multiple HNN extensions of free groups.  
\end{abstract}

\tableofcontents

\section{Introduction}

Gromov was first to notice a connection between the homotopic properties of asymptotic cones of a finitely generated group and algorithmic properties of the group \cite[Section 5.F]{gr2}.  Gromov asked what kind of fundamental groups can asymptotic cones of finitely generated groups have \cite{gr2}. In particular, he asked whether the following dichotomy is true: the fundamental group of an asymptotic cone of a finitely generated group is always either trivial or of order continuum. One reason for this question was that asymptotic cones of nilpotent groups are simply connected (Pansu, \cite{Pansu}), the same is true for hyperbolic groups since all cones in that case  are $\R$-trees, but asymptotic cones of many solvable non-nilpotent groups (say, the Baumslag-Solitar group $BS(2,1)$ or Sol)  contain Hawaiian earrings, and that seems to be a common property of very many other groups
\cite{Burillo},\cite{ConnerKent}.

Answering Gromov's  question about possible fundamental groups of asymptotic cones, Erschler and Osin showed that every countable group is a subgroup of the fundamental group of an asymptotic cone of a finitely generated group\cite{EO}. Dru\c tu and Sapir proved that, moreover, for every countable group $C$, there exists an asymptotic cone of a finitely generated group $G$ whose fundamental group is the free product of uncountably many copies of $C$ \cite{DS}. (Note that for finitely presented groups $G$, analogs of the results of Erschler-Osin and Dru\c tu-Sapir are still unknown.)

It turned out that Gromov's dichotomy is  false: there exists an asymptotic cone of a finitely generated group whose fundamental group is $\Z$ because the cone is homeomorphic to the direct product of a tree and a circle \cite{OOS}.

A \emph{prairie group} is a group where every asymptotic cone of the group is simply connected.  A group is \emph{constricted} if all of its asymptotic cones have (global) cut-points and \emph{wide} if none of its asymptotic cones have cut-points.  We show that for constricted groups Gromov's dichotomy does hold and that a modified version of Gromov's dichotomy holds for groups which are not wide.

\begin{thm*}[\ref{notwide}]
Let $G$ be a finitely generated group.  If $G$ is constricted, then every asymptotic cone of $G$ is simply connected or contains an uncountably generated free subgroup.  If $G$ is not wide, then $G$ has an asymptotic cone  which is simply connected or contains  an uncountably generated free subgroup.
\end{thm*}

The proof requires that we study complements of bounded sets.  While doing this, we obtain the following two results which are interesting in their own right.

\begin{prop*}[\ref{one-ended}]
Let $X$ be a homogeneous geodesic metric space.  Every asymptotic cone of $X$ is one-ended if and only if $X$ is wide if and only if no asymptotic cone of $X$ has a local cut-point.
\end{prop*}

\begin{prop*}[\ref{embedd}]
Suppose that $X$ is an unbounded homogeneous geodesic metric space and $C_i$ is a sequence of finite point sets from $\con$.  Then $\lim^\omega_e C_i$ embeds isometrically into $\con$.
\end{prop*}

A set $T\subset X$ is \emph{transversal} in $X$, if the connected components of $T\backslash \{t\}$ are contained in distinct connected components of $X\backslash \{t\}$ for every $t\in T$.  We use \ref{embedd} to show that every asymptotic cone with a cut-point contains isometrically embedded universals $\mathbb R$-trees which are transversal in the cone.  For groups, this result can be derived from \cite{SistoTree}.

Our methodology to prove these results was to understand bounded sets which separate and how they can arise in asymptotic cones of homogenous spaces.  We were able to extend these methods to subsets which are unbounded to show that Gromov's dichotomy holds for HNN-extensions and amalgamated products with nicely embedded associated subgroups.

\begin{thm*}[\ref{HNN-Amalg}]
Suppose that $G$ is an HNN-extension or amalgamated product where the associated subgroups are  proper, quasi-isometrically embedded, prairie groups.  Then every asymptotic cone of $G$ is either simply connected or has uncountable fundamental group.
\end{thm*}

Recently, work of Dani Wise has brought to light the usefulness of considering groups which have a quasi-convex hierarchy.

Another weaker version of Gromov's dichotomy holds for multiple HNN extensions of free groups:

\begin{thm*}[\ref{HNN-free_associated}]
If $G$ is a multiple HNN extensions of a free group, then every asymptotic cone of $G$ is simply connected or $G$ has an asymptotic cone with uncountable fundamental group.
\end{thm*}

Multiple HNN extensions of free groups can have unusual asymptotic properties.  Burillo in \cite{Burillo} showed that asymptotic cones of HNN extensions of free groups (Baumslag-Soliar groups) can have $\pi_1$-embedded Hawaiian earring groups.  Olshanskii and Sapir in \cite{OS2} construct a multiple HNN extension of a free group which has $\pi_1$-non-equivalent asymptotic cones.

Note that when Gromov's dichotomy was formulated, examples  of groups with several non-homeomorphic (or moreover $\pi_1$-non-equivalent) asymptotic cones were not known. Now we know that a finitely generated group can have uncountably many pairwise $\pi_1$-non-equivalent asymptotic cones \cite{DS} (or much more pairwise non-homeomorphic cones, if the Continuum Hypothesis is assumed false \cite{KSTT}), so our weaker version of the dichotomy seems natural.

\subsection{Definitions}

Let $G=\langle S\rangle$ be a group and $u,v$ be two words in the alphabet $S$. We write $u\equiv v$ when $u$ and $v$ coincide letter by letter and $u=_G v$ if $u$ and $v$ are equal in $G$.  We will denote the Cayley graph of $G$ with respect to the generating set $S$ by $\Gamma(G,S)$.  We will use the standard convention of considering $G$ as the set of vertices of the Cayley graph which acts on $\Gamma(G,S)$ by isometries. The Cayley complex of $G$ depends on a set of generators and a set of relations but will be denoted $\Gamma^2(G,S)$ when the relators are understood. We will use $\operatorname{\textbf{Lab}}$ to represent the function from the set of edge paths in a labeled oriented CW complex to the set of words in the alphabet obtained by reading the label of a path.

We will use $\theta$ to denote the canonical map taking a van Kampen diagram into the Cayley complex which restricts to a label preserving map on the $1$-skeleton of the diagram.  Explicitly, let $\Delta$ be a van Kampen diagram with a distinguished vertex $o$ and $g_o$ a vertex of $\Gamma(G,S)$.  For $v$ a vertex of $\Delta$, let $\theta(v)= g_0w_v$ where $w_v$ is the label of any path in $\Delta$ from $o$ to $v$.   Whenever $\Delta$ is a simply connected diagram, this map is independent of the choice of $w_v$ and extends to a map on all of $\Delta$ as follows.  For $e$ an edge of $\Delta$ labeled by $s$ with initial vertex $v$, let $\theta(e)= (\theta(v),s)$ where $\bigl(\theta(v),s\bigr)$ is the edge in $\Gamma(G,S)$ with initial vertex $\theta(v)$ and labeled by $s$.  For $\pi$ a 2-cell of $\Delta$,  we may choose a vertex $v$ on $\partial \pi$ such that $\lab {\pi} \equiv r^{\pm 1}$  with this choose of base point.  Then $\theta(\pi)= D_{\theta(v), r^{\pm 1}}$ where $D_{\theta(v), r^{\pm 1}}$ is the two cell in $\Gamma^2(G,S)$ with boundary, read from $\theta(v)$, labeled by $r^{\pm1}$.  The map $\theta$ is unique up to our choice of $g_0$ and $o$.

\begin{defn}  Suppose that $\beta$ is a simple closed curve contained in the interior of a planar disc $D$.  Then $D\backslash \beta$ has exactly two components.  The component of $D\backslash \beta$ whose closure contains $\partial D$ will be called the unbounded component of $D\backslash \beta$.  The other component will be called the bounded component.  A point $v\in D$ is \emph{interior (or exterior)} to $\beta$, if it is contained in the bounded (or unbounded) component of $D\backslash \beta$.
\end{defn}

\begin{defn}[Asymptotic cones]
Let $\omega$ be an ultrafilter on $\N$ and  $c_n$ be a sequence of positive real numbers.  The sequence $c_n$ is \emph{bounded $\omega$-almost surely} or \emph{$\omega$-bounded}, if there exists a number $M$ such that $\omega\bigl(\{n \ | \ c_n< M\}\bigr) = 1$.  If $c_n$ is $\omega$-bounded, then there exists a unique number, which we will denote by $\lim^\omega c_n$, such that  $\omega\bigl(\{ n \  | \ |c_n- lim^\omega c_n|< \epsilon\}\bigr) = 1$ for every $\epsilon>0$.

If $c_n$ is not $\omega$-bounded, then $\omega\bigl(\{n \ | \ c_n> M\}\bigr) = 1$ for every $M$.  We will say that \emph{$c_n$ diverges $\omega$-almost surely} or is \emph{$\omega$-divergent} and let $\lim^\omega c_n = \infty$.

Let $(X_n,\d_n)$ be a sequence of metric spaces and $\omega$ an ultrafilter on $\N$.  Consider a sequence of points $e=(e_n)$ such that $e_n \in X_n$ called an \emph{observation sequence}.

Given two elements $x = (x_n), y = (y_n)\in \prod X_n$, set $\d(x,y) = \lim{}^\omega\d_n(x_n,y_n)$.  We can then define an equivalence relation $\sim$ on $\prod X_n$  by  $x\sim y$ if and only if $\d(x,y)= 0$.

The \emph{ultralimit} of $ X_n$ relative to the observation sequence $e$ is
$$\lim{}^\omega_e X_n =\Bigl\{x= (x_n)\in\prod X_n \ | \  \ \d(x,e)<\infty\Bigr\}/\sim.$$

Now consider an $\omega$-divergent sequence of numbers $d = (d_n)$ called a \emph{scaling sequence} and a metric space $(X, \d)$.

The asymptotic cone of $X$ with respect to $e$, $d$, and $\omega$ is

$$ \operatorname{Con}^\omega\bigl(X,e,d\bigr) = \lim{}^\omega_e \bigr(X, \d/d_n\bigr)$$
where $\d/d_n$ is the metric on $X$ scaled by $\frac{1}{d_n}$.

$\operatorname{Con}^\omega\bigl(X,e,d\bigr)$ is a complete metric space. If $X$ is geodesic, then $\operatorname{Con}^\omega\bigl(X,e,d\bigr)$ is also geodesic.  If $X$ is a homogeneous metric space, then the isometry type of $\con$ is independent of $e$ and will frequently  be denoted by $\operatorname{Con}^\omega\bigl(X,d \bigr)$.

Suppose that $\{X_n\}$ is a sequence of subsets of a metric space $(X, \d)$.  At times it will be convenient to talk about the subset of $\con$ with representatives in $\prod X_n$.  When it is clear from the text, we will denote this subset by $\lim^\omega X_n$ instead of, the more precise, $\lim_e^\omega (X_n, \d/d_n)$.  When used in this context, we will not require that $e_n$ be an element of $X_n$.

\end{defn}

\section{Wide groups and ends of asymptotic cones}\label{wide_section}

The following lemma is obvious.

\begin{lem}
Let $\omega$ be an ultrafilter on $\N$ and $d=(d_n)$ a scaling sequence. Suppose that $\{\gamma_n\}$ is a sequence of loops parameterized by arc length in a geodesic metric space $(X, \d)$ such that $|\gamma_n| = O(d_n)$. Then $\gamma(t) = \bigl(\gamma_n(t)\bigr)$ is a continuous map of $S^1$ into $\con$.
\end{lem}

The converse also holds.

\begin{lem}\label{limitloop}\ulabel{limitloop}{Lemma}
Suppose that $\gamma$ is a path in $\con$ where $X$ is a geodesic metric space.  Then there exist paths $\gamma_n$ in $X$ such that $\gamma(t) = \bigl(\gamma_n(t)\bigr)$.
\end{lem}

Recall that there exists geodesics in a cone which are not limits of geodesics.  However, here we do not put any restraints on the paths $\gamma_n$ (the proof shows that $\gamma_n$ can be chosen to be a $2^n$-gon).

\begin{proof}
Suppose that $\gamma\co [0,1]\to \con$ is a path and let $\xi$ be a modulus of continuity for $\gamma$.

For each diadic rational $r$, fix a representative $\bigl(a_n(r)\bigr)$ of $\gamma(r)$. Define $\gamma_n\co [0,1] \to X$ by

$$\gamma_n(r)= a_n(r) \text{ for } r\in \Bigl\{ 0, \frac {1}{2^n}, \cdots , \frac{2^n -1}{2^n}, 1\Bigr\}$$

and extend $\gamma_n$ geodesically.

We can then define $\gamma'(t) = \bigl(\gamma_n(t)\bigr)$  and notice that $\gamma'(t) = \gamma(t)$ on the diadic rationales.  Suppose that $x_i\in [0,1]$ and $x_i$ converges to $x_0$.  For each $i$, let $r_i, s_i$ be closest points of $\bigl\{ 0, \frac {1}{2^i}, \cdots , \frac{2^i -1}{2^i}, 1\bigr\}$ to $x_i, x_0$ respectively.  Then $\d\bigl(\gamma_n(x_i), \gamma_n(r_i)\bigr), \d\bigl(\gamma_n(x_0), \gamma_n(s_i)\bigr)\leq \xi(\frac{1}{2^i})d_n + o(d_n)$.  Hence $\d\bigl(\gamma'(x_i), \gamma'(r_i)\bigr), \d\bigl(\gamma'(x_0), \gamma'(s_i)\bigr) \leq \xi(\frac{1}{2^i})$.

For each $i$,
\begin{align*}
\d\bigl(\gamma'(x_i), \gamma'(x_0)\bigl) & \leq \d\bigl(\gamma'(x_i), \gamma'(r_i)\bigr) +\d\bigl(\gamma'(r_i), \gamma'(s_i)\bigr) + \d\bigl(\gamma'(s_i), \gamma'(x_0)\bigr) \\ &\leq 2\xi\bigl({1}/{2^i}\bigr) + \xi\bigr(\d(r_i, s_i)\bigl).
\end{align*}
Since $\d(r_i,s_i)\leq \frac{1}{2^{i-1}} + \d(x_i, x_0)$, this implies that $\gamma'$ is continuous.  Thus $\gamma = \gamma'$.

\end{proof}

\begin{lem}\label{limitdisc}\ulabel{limitdisc}{Lemma}
Suppose that $h\co A\to \con$ is a continuous map where $A$ is any compact planar Peano continuum and $X$ is a simply connected geodesic metric space which satisfies a linear isodiametric function.  Then there exist continuous maps $h_n\co A \to X$ such that $\gamma(t) = \bigl(\gamma_n(t)\bigr)$.
\end{lem}

The proof is essentially the same as for the previous lemma and uses the covering from \ref{opensquares}.

Throughout this paper, we will assume that metric balls are open.  When $\tau$ is a path in a metric space, we will use $|\tau|$ to denote its arc length. Then $|\cdot|$ maps the set of paths into the extended real line and is finite for rectifiable paths and $+\infty$ for non-rectifiable paths. We will assume that rectifiable paths are parameterized proportional to arc length.  We will use $\mathcal N_s(B)$ to represent the $s$-neighborhood of $B$.

\begin{lem}\label{unbounded}\ulabel{unbounded}{Lemma}
Let $X$ be a homogeneous geodesic metric space.  The non-empty components $\con\backslash\bigl\{(x_n)\bigr\}$ are unbounded for all $(x_n)\in\con$.
\end{lem}

\begin{proof}  The lemma is trivial if $X$ is bounded.  Every asymptotic cone of an unbounded geodesic metric space contains a bi-infinite geodesic.  Hence we may assume $\con$ contains a bi-infinite geodesic through every point.  Suppose that $A$ is a non-empty connected component of $\con\backslash \{x\}$ for some $x\in\con$.  Let $a\in A$.  Then there exists a bi-infinite geodesic $\alpha\co \mathbb R \to \con$ such $\alpha(0) = a$.  Only one of $\alpha\bigl( (-\infty, 0]\bigr)$, $\alpha\bigl([0,\infty)\bigr)$ can intersect $x$.  Hence $A$ must contain an unbounded ray. \end{proof}

This lemma also follows from \cite[Lemma 3.12]{DMS}.

\begin{lem}\label{limitseparates}\ulabel{limitseparates}{Lemma}
Let $B_n$ be a sequence of uniformly bounded subsets of a geodesic metric space $X$ and $\kappa\co\mathbb N\to \mathbb R$ be a sublinear function.  If $X\backslash \mathcal N_{k(d_n)}(B_n)$ has more than one unbounded connected component; then, for $e_n\in B_n$, $\con\backslash \lim^\omega_e B_n$ has more than one unbounded connected component.
\end{lem}

\begin{proof}
Let $\{U_{n,1},\cdots, U_{n, i_n}\}$ be the set of unbounded connected components of $X\backslash\mathcal N_{k(d_n)}(B_n)$.  Let $B = \lim_e^\omega B_n$, $Z=  \lim_e^\omega U_{n, s_n}$, and $Y = \lim_e^\omega U_{n,t_n}$ where $s_n, t_n$ are distinct elements of $\{1,\cdots, i_n\}$ \as.  Since $B_n$ is uniformly bounded and $U_{n,i}$ is unbounded, both $Y\backslash B$ and $Z\backslash B$ are nonempty and hence unbounded.

Suppose that $x\in Z\cap Y$.  Then $x = (z_n)= (y_n)$ where $z_n \in U_{n, s_n}$ and  $y_n \in U_{n, t_n}$.  Since $U_{n, s_n}$ and  $U_{n, t_n}$ are in distinct connected components of $X\backslash\mathcal N_{k(d_n)}(B_n)$, every path originating in $U_{n,s_n}$ and terminating in $U_{n, t_n}$ passes through $\mathcal N_{k(d_n)}(B_n)$.  By considering a geodesic from $z_n$ to $y_n$, we can find $b'_n \in\mathcal N_{k(d_n)}(B_n)$ such that $\d(z_n, b'_n)+\d(b'_n, y_n)= \d(z_n, y_n)$  which implies that $x = (b'_n)$.  As well, there exists $b_n\in B_n$ such that $\d(b_n, b_n')\leq \kappa(d_n)$.  Hence $x = (b_n') = (b_n)$.

Thus $Z\cap Y \subset B$ and the components of $Z\backslash B, Y\backslash B$ are distinct unbounded components of $\con\backslash B$.

\end{proof}

At times it will be convenient to consider separating sets which are unbounded.

\begin{lem}\label{limitseparates2}\ulabel{limitseparates2}{Lemma}
Let $(B_n, e_n)$ be a sequence of pointed subsets of a geodesic metric space $X$, $\omega$ an ultrafilter, and $d=(d_n)$ an $\omega$-divergent sequence.  Suppose that $a= (a_n), b= (b_n) \in \con$ are points such that there exists a sublinear function $\kappa\co\mathbb N\to \mathbb R$ such that $a_n, b_n$ are in distinct components of $X\backslash \mathcal N_{\kappa(d_n)}(B_n)$ $\omega$-almost surely.

Then $ \lim_e^\omega B_n$ separates $\con$ into at least two connected components and $a$, $b$ are in distinct components of $\con\backslash \lim_e^\omega B_n$.
\end{lem}

The proof is the same as for bounded sets; the only difference is that we are not able to conclude that the components are unbounded since we cannot apply \ref{unbounded}.

\begin{defn}
Let $X$ be a connected, locally connected topological space.  A point $x\in X$ is a \emph{local cut-point} if there exists an open connected neighborhood $U$ of $x$ such that $U\backslash\{x\}$ has at least two connected components.  A point $x\in X$ is a \emph{global cut-point} if $X\backslash\{x\}$ has at least two connected components.  $X$ is \emph{wide} if none of its asymptotic cones has a global cut-point. $X$ is \emph{unconstricted} if one of its asymptotic cones has no global cut-points.  $X$ is \emph{constricted} if all of its asymptotic cones have global cut-points.

Let $B_1\subset B_2\subset \cdots$ be an ascending sequence of bounded sets in a metric space $X$ such that every set of bounded diameter is eventually contained in $B_n$ for some $n$.  This implies that $\cup_n B_n = X$.

Two descending sequences $ U_1\supset U_2\supset \cdots$ and $V_1\supset V_2\supset \cdots$ of subsets of $X$ are \emph{equivalent} if for every $n$ there exists integers $m,k$ such that $V_m\subset U_n$ and $U_k\subset V_n$.

An \emph{end of $X$} is a descending sequence $U_1\supset U_2\supset \cdots$ where $U_i$ is an unbounded component of $X\backslash B_i$.  It can be shown that up to the given equivalence on descending sequences of subsets of $X$ the set of ends of $X$ does not depend on $\{B_n\}$.

A metric space $X$ is \emph{one-ended}, if $X\backslash B$ has a unique unbounded connected component for every bounded subset $B$ of $X$.

\end{defn}
We will use the following definition and lemma from \cite{DMS}.

\begin{defn}\label{div} Let $X$ be a geodesic metric space, and let $0<\delta<1$ and $\gamma \geq 0$. Let $a,b,c\in X$ with $\d\bigl(c,\{a,b\}\bigr)=r>0$, where $\d(c,\{a,b\})$ is the minimum of $\d(c,a)$ and $\d(c,b)$. Define $\dv_{\red{\gamma }}(a,b,c;\delta)$ as the infimum of the lengths of paths $a, b$ that avoid the ball $\Ball(c,\delta r -\gamma )$ (note that by definition a ball of non-positive radius is empty). If no such path exists, take $\dv_{\gamma }(a,b,c;\delta)=\infty$.
\end{defn}

\begin{lem}[{\cite[Lemma 3.14]{DMS}}]\label{divcondition}\ulabel{divcondition}{Lemma}
Let $X$ be a geodesic metric space. Let $\omega$ be any ultrafilter and  $d=(d_n)$ be an $\omega$-divergent sequence. Let $a=(a_n), b=(b_n), c=(c_n)\in \con$. Let $r=\d\bigl(c, \{a,b\}\bigr)$. The following conditions are equivalent for any $0\le \delta<1$.

\begin{enumerate}[(i)]
\item The closed ball $\overline{\Ball}(c,\delta)$ in $\con$ separates $a$ from $b$.
\item For every $\delta'>\delta$ and every (some) $\gamma \geq 0$ the limit $\lim^\omega \frac{\dv_{\gamma }(a_n,b_n,c_n;\frac{\delta'}{r})}{d_n}$ is $\infty$.
\end{enumerate}

\end{lem}

The following proposition is immediate, as it holds for all homogeneous geodesic metric spaces, see \cite{Geoghegan}.

\begin{prop} An asymptotic cone of a finitely generated group can have 0,1,2 or uncountable many ends. \end{prop}

\begin{lem}
Let $X$ be a homogenous geodesic metric space.  If $\con$ has a local cut-point, then there exists a cone of $X$ with a global cut-point.
\end{lem}

\begin{proof}
Suppose that $\con$ has a local cut-point. By homogeneity, $\tilde x = (x_n)$ is a local cut-point.  Suppose that $U$ is an open connnected neighborhood of $\tilde x$ such that $U\backslash \{\tilde x\}$ has two components.

\begin{clm*}
There exists an $\epsilon>0$ such that $\tilde x$ separates every ball about $\tilde x$ with radius at most $\epsilon$.
\end{clm*}

Let $B_\epsilon$ be the ball in $\con$ about $\tilde x$ of radius $\epsilon$.

Fix $\epsilon>0$ such that $B_\epsilon$ is a subset of $U$.  Let $u,v$ be elements of $U$ which are in different components of $U\backslash \{\tilde x\}$. Any path in $U$ from $u$ to $v$ passes through $\tilde x$.  (Since $\con$ is locally path connected and $U$ is open and connected, $U$ is path connected.) Hence, we can find a path $f\co[0,1]\to U$ such that $f^{-1}(\tilde x) = \{\frac12\}$ and $f(0),f(1)$ are in different components of $U\backslash\{\tilde x\}$.  This implies that the inclusion map from $B_{\epsilon'}\backslash \{\tilde x\}$ to $U\backslash \{\tilde x\}$ is not contain in a single component for any $\epsilon'\leq \epsilon$.  Thus $B_{\epsilon'}\backslash \{\tilde x\}$ is also not connected for any $\epsilon'\leq \epsilon$ which completes the proof of the claim.

We can now consider the cones $X_k^\omega= \ccon{X}{e_n}{\frac{d_n}{k}}$.  It is easy to see that $\tilde x$ is a cut-point of the ball of radius $k\epsilon$ in $X_k$.  Hence, $\lim^\omega X_k$ has a global cut point and by \cite[Corollary 3.24]{DS} $\lim^\omega X_k$ is again an asymptotic cone of $X$.

\end{proof}

\begin{prop}\label{one-ended}\ulabel{one-ended}{Proposition}
Let $X$ be a homogeneous geodesic metric space.  Every asymptotic cone of $X$ is one-ended if and only if $X$ is wide if and only if no asymptotic cone of $X$ has a local cut-point.
\end{prop}

\begin{proof}
$X$ is wide if and only if no asymptotic cone of $X$ has a local cut-point follows immediately from the previous lemma.  Thus we need only proof that every asymptotic cone of $X$ is one-ended if and only if $X$ is wide.

The only if direction of this equivalence is trivial.  We must show that if no asymptotic cone of $X$ has a cut-point, then every asymptotic cone of $X$ is one-ended.  Suppose that no asymptotic cone of $X$ has a cut-point but $\con$ is not one-ended for some choice of $\omega, e, d$.  Hence, there exists a bounded subset $\tilde B$ of $\con$ such that $\con\backslash \tilde B$ has at least two unbounded components.  By homogeneity, we may assume that $\tilde x = (x_n) \in \tilde B$.

By \ref{limitseparates}, $\widetilde Y= \ccon{\con}{\tilde x}{n}$ ${}\backslash \lim^\omega B$ has more than one connected component.  Since $\tilde B$ is bounded, $\lim^\omega \tilde B$ is a point in $\widetilde Y$ which separates.  Thus it is a cut-point of $\widetilde Y$.  $\widetilde Y$ is again an asymptotic cone of $X$ \cite[Corollary 3.24]{DS}.  This contradicts the hypothesis that no cone of $X$ has a cut-point.

\end{proof}

In \cite[Theorem 1.4]{DMS}; Drutu, Mozes, and Sapir show that certain semisimple Lie groups (namely those specified in the theorem below) are wide.  Hence, we can apply \ref{one-ended} to obtain the following result.

\begin{thm}
Let $\Gamma$ be an irreducible lattice in a semisimple Lie group of $\mathbb R$-rank ≥ 2. Suppose that $\Gamma$ is either of $\mathbb Q$-rank 1 or is of the form $SL_n(\mathcal {O_S})$ where $n \geq 3$, $\mathcal S$ is a finite set of valuations of a number field $K$ including all infinite valuations, and $\mathcal {O_S}$ is the corresponding ring of $\mathcal S$-integers. Then every asymptotic cone of $\Gamma$ is one-ended.
\end{thm}

\ref{one-ended} together with \cite[Corollary 6.13]{DS} give us the following.

\begin{thm}
Let $G$ be a finitely generated non-virtually cyclic group satisfying a law.  Then all asymptotic cones of $G$ are one-ended.
\end{thm}

The following proposition is a well know.  We present it here only for comparison with \ref{cyclic}.

\begin{prop}
Let $G$ be a finitely generated group.  The following are equivalent:
\begin{enumerate}[a)]

    \item\label{A} $G$ is finite.
    \item\label{B} $G$ has an asymptotic cone which is a point.
    \item\label{C} $G$ has an asymptotic cone with 0 ends.

\end{enumerate}
\end{prop}

\begin{proof}
If $G$ is an infinite finitely generated group, then $\Gamma(G,S)$ contains a bi-infinite geodesic for every finite generating set $S$.  Thus $\gcon$ contains a bi-infinite geodesic for every infinite group $G$.  If $G$ is finite then $\Gamma(G,S)$ is bounded for every generating set $S$ and $\gcon$ is a point for every pair $\bigl(\omega, d\bigr)$.  Thus (\ref{A}) and (\ref{B}) ar equivalent.  Clearly, (\ref{B}) implies (\ref{C}).  If $\gcon$ has 0 ends for some pair $\bigl(\omega, d\bigr)$, then it doesn't contain a bi-infinite geodesic.  Hence (\ref{C}) implies (\ref{A}).

\end{proof}

\begin{prop}\label{cyclic}\ulabel{cyclic}{Proposition}
Let $G$ be a finitely generated group.  The following are equivalent:
\begin{enumerate}[a)]

    \item\label{A'} $G$ is infinite and virtually cyclic.
    \item\label{B'} $G$ has an asymptotic cone which is a line.
    \item\label{C'} $G$ has an asymptotic cone with exactly 2 ends.

\end{enumerate}
\end{prop}

\begin{proof}
If $G$ is infinite and virtually cyclic, then $\gcon$ is a line for every pair $\bigl(\omega, d\bigr)$.  Thus (\ref{A'}) implies (\ref{B'}).  The implication $(\ref{B'}) \Rightarrow (\ref{A'})$ is Corollary 6.2 in \cite{DS}; as well, it also follows from \cite{Point}, since a line has finite Minkowski dimension.

Thus we need only show that if $\gcon$ has exactly two ends for some pair $\bigl(\omega, d\bigr)$ than $G$ has an asymptotic cone which is a line.

Suppose that $\gcon\backslash B_\rho(x_0) $ has exactly two unbounded components for some $\rho>0$ and $x_0\in \gcon$.  For each $i$, let $U_i$ and $ V_i$ be the two unbounded components of $\gcon\backslash B_{i\rho}(x_0)$.  We may assume that we have chosen $U_i, V_i$ such that $U_i\supset U_{i+1}$ and $V_i\supset V_{i+1}$ for all $i$.  Fix $x_i\in U_i$ and $x_{-i}\in V_i$ such that $\d(x_0, x_{\pm i}) = i\rho$ for all $i\in\mathbb N$.

Define a path $\alpha\co \mathbb R \to \gcon$ by $\alpha(i) = x_i$, for $i \in \mathbb Z$, and for every $i\in \mathbb Z$ extend $\alpha$ to $[i, i+1]$ by sending the interval to a geodesic joining its endpoints.

\begin{clm}\label{clm1}\ulabel{clm1}{Claim}  $\alpha$ is a quasi-geodesic with constants depending only on $\rho$ and $\gcon$ is contained in the $2\rho$-neighborhood of the image of $\alpha$.\end{clm}

Notice that \ref{clm1} implies that $G$ has an asymptotic cone which is a line since any asymptotic cone of $\gcon$ is a line and an asymptotic cone of $G$.

Let $\alpha_i^-=\alpha\bigr((-\infty, i-4]\bigl)$, $\alpha_i^+ = \alpha\bigl([i+4, \infty)\bigr)$ and $Y_i = \gcon\backslash B_\rho(x_i)$ for all $i$.  By homogeneity, $Y_i$ has exactly 2 unbounded connected components

\begin{sclm}\ulabel{sclm1}{Subclaim} For all $i,j\in \mathbb Z$, $\d(x_i, x_j)\geq \bigl|j-i\bigr|\rho -2\rho$ and hence $\alpha_i^\pm\subset Y_i$.
\end{sclm}

\begin{proof}[Proof of \ref{sclm1}.]
If $i,j$ have the same sign then by applying the triangle inequality to a geodesic triangle with vertices $x_0, x_i, x_j$, we obtain $\d(x_i, x_j)\geq |j-i|\rho$.

Suppose that $i\leq 0 \leq j$.  By construction, every geodesic from $x_i$ to $x_j$ passes within $\rho$ of $x_0$.  Fix a geodesic from $x_i$ to $x_j$ and let $x_0'$ be a point on the geodesic such that $\d(x_0, x_0') \leq \rho$.  Then $-i\rho = \d(x_i, x_0) \leq \d(x_i, x_0') + \rho$ and $j\rho = \d(x_0, x_j) \leq \d(x_0', x_j) + \rho$ which gives us that $(j-i)\rho\leq \d(x_i, x_0')+ \d(x_0', x_j) + 2\rho = \d(x_i,x_j) + 2\rho$.  Thus $\d(x_i, x_j)\geq |j-i|\rho -2\rho$.  If $j\in (\infty,i-4]\cup [i+4, \infty)$, then $\d(x_i,x_j)\geq 2\rho$. Since every point on $\alpha_i^\pm$ is with in $\rho$ of some $x_j$ for $j\in (\infty,i-4]\cup [i+4, \infty)$; $\alpha_i^\pm\subset Y_i$.
\end{proof}

\begin{sclm}\ulabel{sclm2}{Subclaim}  $\alpha_i^+,\alpha_i^-$ are contained in distinct unbounded components of $Y_i$ for all $i$.\end{sclm}

\proof[Proof of \ref{sclm2}.]
We will show the subclaim for $i\geq 0$.  The other case is similiar.  Let $U,V$ be the two disjoint unbounded components of $Y_i$.  By way of contradiction, we will assume that $\alpha_i^\pm$ are both contained in $U$.  Choose $\tilde g \in \prod G$ such that $\tilde g\cdot x_0 = x_i$.  For each $j \geq 1$, let $\beta_j^-=\alpha\bigr((-\infty, -2j-4]\bigl)$, $\beta_j^+ =\alpha\bigl([2j+4, \infty)\bigr)$.  Since $\tilde g$ acts by isometries on $\gcon$, we obtain that $Y_i=\tilde g\cdot  Y_0$ and $\tilde g\cdot\beta_{j}^\pm$ are in distinct unbounded components of $Y_i$ for any $j\geq 1$.

Fix $j\geq i$.  Since $\alpha_i^-,\alpha_i^+$ are contained in the same connected component of $Y_i$ and $\tilde g\cdot\beta_j^-,\tilde g\cdot\beta_j^+$ are contained in distinct connected components of $Y_i$, one of $\tilde g\cdot\beta_j^\pm$ is contained in $V$.  Suppose that $\tilde g\cdot\beta_j^+\subset V$. (Again the other case is similar.)  Notice that $\beta_j^\pm\subset \alpha_i^\pm$ which implies that $\beta_j^-\cup\beta_j^+ \subset U$.

By \ref{sclm1}, $\d(x_i, \tilde g\cdot\beta_j^\pm) \geq (2j+4)\rho-2\rho$ which implies that $\d(x_0, \tilde g\cdot\beta_j^\pm)\geq (j+2)\rho$.  Thus $\tilde g\cdot\beta_j^\pm \subset \gcon\backslash B_{(j+1)\rho}(x_0)$.   Again by \ref{sclm1}, $\beta_j^\pm \subset \gcon\backslash B_{(j+1)\rho}(x_0)$.  Since $B_\rho(x_0), B_\rho(x_i)\subset B_{(j+1)\rho}(x_0)$, each of the three unbounded sets $\beta_j^\pm, \tilde g\cdot\beta_j^+$ must be contained in a distinct connected component of $\gcon\backslash B_{(j+1)\rho}(x_0)$.  Since this holds for any $j\geq i$, $\gcon$ must have at least 3 ends which contradicts our assumption that $\gcon$ has exactly 2 ends.
\endproof

\begin{proof}[Proof of \ref{clm1}]  If $i,j\in \mathbb Z$ have different signs, then $\d(x_i,x_j) \leq \d(x_i, x_0)+ \d(x_j, x_0) = |i|\rho +|j|\rho = |i-j|\rho$

For $4\leq i\leq j-4$, any geodesic from $x_0$ to $x_j$ is passes within $\rho$ of $x_i$ by \ref{sclm2}.  Hence, we may find a point $x_i'$ on a geodesic from $x_0$ to $x_j$ such that $\d(x_i, x_i')\leq \rho$.  Then $i\rho\leq \d(x_0, x_i') + \rho$ and $\d(x_i, x_j) \leq \d(x_i', x_j) +\rho$ which implies that $\d(x_i, x_j)\leq (j-i)\rho +2\rho= |j-i|\rho +2\rho$.   Similarly, we can obtain the inequality $\d(x_i, x_j)\leq |j-i|\rho +2\rho$ for  $j+4\leq i \leq -4$.  It follows that $\alpha$ is a quasi-geodesic.

Suppose that there exists $x\in \gcon$ such that $\d(x, \im{\alpha})\geq 2\rho$.  By homogeneity, $\gcon\backslash B_\rho(x)$ has two unbounded components one of which contains $\im{\alpha}$.  As in the proof of \ref{sclm2}, this would imply that $\gcon$ would have at least three ends.
\end{proof}

Thus any asymptotic cone of $\gcon$ is a line and also an asymptotic cone of $G$ which completes the proof of the proposition.
\end{proof}

\begin{lem}\label{tree}\ulabel{tree}{Lemma}
Suppose that $X$ is an unbounded homogeneous geodesic metric space and $T$ is a vertex homogeneous three valence tree with fixed edge length $\rho$.  If $\con$ has more than two ends and a global cut-point, then there exists an isometry $f\co T\to \con$  such that the components of $T\backslash\{v\}$ map to distinct components of $\con\backslash \bigl\{f(v)\bigr\}$ for every vertex $v$ of $T$.
\end{lem}

\begin{proof}Fix $\rho>0$. Let $T$ be a vertex homogeneous 3-valence tree with fix edge length $\rho$.  We will now build an isometry $f\co T \to \con$ such that the three components of $T\backslash\{v\}$ map into distinct components of $\con\backslash\bigl\{f(v)\bigr\}$ for every vertex $v$ of $T$.  Fix a vertex $v_0$ of $T$.

Let $T_i$ be a sequence of connected subtrees of $T$ such that $v_0 = T_1$; $T_i\subset T_{i+1}$; $\cup_i T_i = T$; and $T_{i+1}$ has exactly one vertex not contained in $T_i$.  This implies that $T_{i+1}$ can be obtained from $T_i$ be adding exactly one edge and one vertex.

Let $f(v_0) = x_0$ for some $x_{0}\in\con$.  By induction, assume that we have defined $f$ on $T_i$ such that $f|_{T_i}$ is an isometry and $f$ maps the components of $T_i\backslash\{v\}$ to distinct components of $\con\backslash \{f(v)\}$ for each vertex $v$ of $T_i$.  Let $e$ be the edge of $T$ which is added to $T_i$ to obtain $T_{i+1}$.  Then $e$ has exactly one vertex $e^-$ in $T_i$ and one vertex $e^+$ in $T_{i+1}\backslash T_i$.  Notice that $T_i$ has valence 1 or 2 at $e^-$.  This implies that $T_i\backslash\{e^-\}$ and hence $f\bigr(T_i\backslash\{e^-\}\bigl)$ has at most 2 components.  Let $C$ be a component of $\con\backslash \{f(e^-)\}$ which is disjoint from $f\bigr(T_i\backslash\{e^-\}\bigl)$.  Since all components are unbounded, we may choose a point $x\in C$ such that $\d\bigl(x, f(e^-)\bigl) = \rho$.  Let $f(e^+) = x$ and $f(e)$ be a geodesic from $f(e^-)$ to $f(e^+)$.  It is immediate that the components of $T_{i+1}\backslash\{v\}$ map to distinct components of $\con\backslash \{f(v)\}$ for all vertices $v$ in $T_{i+1}$.  It only remains to show that $f$ restricted to $T_{i+1}$ is still an isometry.  This follows trivially from the fact that if $x,y$ are in distinct components of $ \con\backslash \{z\}$, then $\d(x,y) = \d(x,z)+ \d(z,y)$.

This defines a map $f\co T \to \con$.  Since any two points lie in some $T_i$, $f$ is an isometry.  We must show that the separation condition is preserved in the limit.  Suppose that $v$ is a vertex of $T$ and $T_i$ contains the $2\rho$-neighborhood of $v$.  By construction, $f$ takes the components of $T_i\backslash \{v\}$ into distinct components of $\con\backslash \{f(v)\}$.  Notice that each component of $T\backslash \{v\}$ intersects a component of $T_i\backslash \{v\}$ nontrivially which implies that the separation condition still holds.

\end{proof}

\begin{cor}\label{corTree}\ulabel{corTree}{Corollary}
In addition, $f$ can be chosen such that  $f(t) = \bigl(f_n(t)\bigr)$ for all $t\in T$ where $f_n\co T\to X$ takes edges of $T$ to geodesics in $X$.
\end{cor}

\begin{proof}
We will show how to modify the proof of \ref{tree}.  Using the notation from above, we will inductively defining $f, f_n$ simultaneously.  Suppose that $f, f_n$ are defined as desired on $T_i$.  When choosing $x\in C$ we will also fix a representative $(x_n)$ of $x$.  Let $f_n(e^+) = x_n$ which implies that $f(e^+) = \bigl(f_n(e^+)\bigr) = x$.  Let $f_n$ map $e$ to any geodesic from $f_n(e^-)$ to $f_n(e^+)$ which implies that $f(e) = \bigl(f_n(e)\bigr)$ is a geodesic from $f(e^-)$ to $f(e^+)$.  The rest of the proof remains unchanged.
\end{proof}

\begin{prop}\label{embedd}\ulabel{embedd}{Proposition}
Suppose that $X$ is an unbounded homogeneous geodesic metric space and $C_i$ is a sequence of finite point sets from $\con$.  Then $\lim_e^\omega C_i$ embeds isometrically into $\con$.  In addition; if $C_i$ is nested, then the canonical copy of $C_i$ in $\lim_e^\omega C_i$ is mapped to $C_i$.
\end{prop}

\begin{proof}
Let $\iota_i\co C_i \to \con$ be the inclusion induced map.  Fix a representative for each element of $C=\cup_i C_i$.  We can now define a double indexed sequence of maps $ \iota_n^i \co C_i \to X$ by letting $\iota_n^i(c)$ be the $n$-th coordinate of our chosen representative for $c\in C$.  Thus, if the $C_i$ are nested and $c\in C_i$; then $\iota_n^j(c) = \iota_n^i(c)$ for all $j\geq i$. Hence, $c = \bigl(\iota_n^{k_i}(c)\bigr)$ for any sequence $k_i$.  This will imply that the map defined below takes the canonical copy of $C_i$ in $\lim^\omega C_i$ to $C_i$. Let

$$A_i = \Bigl\{ n\ \bigl| \ \d(c,c')- \frac{1}{i} \leq \frac{\d\bigl(\iota_n^j(c),\iota_n^j(c')\bigr)}{d_n} \leq \d(c,c') +\frac{1}{i} \text{ for all } c,c'\in C_j \text{ where } j\leq i\Bigr\}.$$

Since $\bigl|\bigcup\limits_{j\leq i}C_j\Bigr|$ is finite and $\iota_j$ is an isometry for every $j$, $A_i$ is $\omega$-large.  Let $m_n = \max \{i \ | \ n\in  A_i \text{ and } i\leq n \}$, if this set is non-empty and $m_n = 1$ otherwise.  Suppose that $m_n$ was bounded by $L$ on some $\omega$-large set $A$.  Then $A_{2L}\cap A \subset \{1, \cdots, 2L-1\}$,  which is a contradiction since $\omega (A_{2L}) = \omega(A) = 1$ and $\omega \bigl( \{1, \cdots, 2L-1\} \bigr) =0 $.  Thus $\lim^\omega m_n = \infty$.

Define $\tilde \iota\co \lim^\omega_e C_i \to \con$ by $\tilde \iota\bigl((c_n)\bigr) = \bigl(\iota_n^{m_n}(c_n)\bigr)$.

\textbf{Claim:} $\tilde \iota$ is a well-defined isometric embedding of $\lim^\omega_e C_n$ into $\con$.

Fix $c,c'\in \lim^\omega C_n$.  We may choose representatives $c_n,c_n'\in C_n$ such that $c = (c_n)$ and $c= (c'_n)$. By construction, $\d(c_n, c'_n) -\frac{1}{m_n} \leq \frac{\d\bigl(\iota_n^{m_n}(c_n),\iota_n^{m_n}(c'_n)\bigr)}{d_n} \leq \d(c_n, c'_n) + \frac{1}{m_n}$  for all $n$ such that $m_n\neq 1$.  Since $m_n$ is $\omega$-divergent, this set is $\omega$-large and
\begin{align*}\d(c,c') &= \lim{}^\omega \Bigl[\d(c_n, c'_n) -\frac{1}{m_n}\Bigr] \leq \lim{}^\omega\Bigl[ \frac{\d\bigl(\iota_n^{m_n}(c_n), \iota_n^{m_n}(c'_n)\bigr)} {d_n} \Bigr] \\ & \leq \lim{}^\omega \Bigl[\d(c_n, c'_n) + \frac{1}{m_n}\Bigr] = \d(c,c').\end{align*}
Thus $\tilde \iota$ is independent of the chosen representative and is an isometry.

\end{proof}

We can now use \ref{tree} to prove that $\mathbb R$-trees can also be transversally embedded into cones with cut-points.

\begin{lem}\label{Rtree}\ulabel{Rtree}{Lemma}
Suppose that $X$ is a unbounded homogeneous geodesic metric space and $T$ is a universal $\mathbb R$-tree with continuum branching at every point.  If $\con$ has more than two ends and a global cut-point, then there exists an isometry $f\co T\to \con$  such that the components of $T\backslash\{v\}$ map to distinct components of $\con\backslash \bigl\{f(v)\bigr\}$ for every $v$ in $T$.
\end{lem}

\begin{proof}
Let $T_i$ be a three valence tree with edge length $\frac{1}{2^i}$ such that $T_i\subset T_{i+1}$ for all $i\in \mathbb N$ and $t_0$ a fixed vertex in $T_1$.  We will assume that $T_i$ is endowed with the edge metric.  We will use $[v,w]$ to denote the geodesic from $v$ to $w$ in $T_i$ and $(v,w) = [v,w]\backslash \{v,w\}$.  If $v,w\in T_i\cap T_j$, then $[v,w]$ is independent of whether the geodesic is taken in $T_i$ or in $T_j$.

By \ref{tree}, there exist isometries $f_i\co T_i \to \con$ which satisfy the separation condition of \ref{tree}. By homogeneity, we may assume $f_i(t_0) = f_j(t_0)$ for all $i,j$.  By \ref{corTree}, there exists a sequence of maps $f_n^i \co T_i \to X$ such that $f_i(t) = \bigl(f_n^i(t)\bigr)$ for all $t\in T_i$.
We will also require that $f_n^i(t_0) = f_n^j(t_0)$ for all $i,j$.

Let $V_i$ be the vertices of the ball of radius $i$ about $t_0$ in $T_i$.  Then $\bigl|f_i(V_i)\bigr|$ is finite set and \ref{embedd} implies $\lim^\omega_e f_i(V_i)$ embeds isometrically.  While $\lim^\omega_e f_i(V_i)$ is a universal $\mathbb R$-tree, we must still guarantee that the embedding preserves the separation property.  To do this we will show how to modify the proof of \ref{embedd} so as to guarantee that the embedding preserves the desired separation property.  Let
$$A_i = \Bigl\{ n\ \bigl| \ \d(v,w)- \frac{1}{i} \leq \frac{\d\bigl(f_n^j(v),f_n^j(w)\bigr)}{d_n} \leq \d(v,w) +\frac{1}{i} \text{ for all } v,w\in V_j \text{ where } j\leq i\Bigr\}.$$

For $r = \d\bigl(f_n^j(v_0), \{f_n^j(v_1),f_n^j(v_2)\}\bigr)$ and $v_0,v_1,v_2 \in T_j$ such that $v_0$ separates $v_1$ from $v_2$ in $T_j$, let $ \rho_n^i(j,v_0,v_1,v_2) = \dv_{1}\bigl(f_n^j(v_1),f_n^j(v_2),f_n^j(v_0);\frac{1}{ir}\bigr)$.  Let
$$B_i = \{ n\ | \ \rho_n^i(j,v_0,v_1,v_2)> id_n \ | \ j\leq i; \text{ } v_0,v_1,v_2\in V_j; \text{ and } v_0\in(v_1,v_2) \}.$$

As before $A_i$ is $\omega$-large for each $i$.

\textbf{Claim:} \emph{$B_i$ is an $\omega$-large set.}\\
For each $j$ and each triple $v_0,v_1, v_2\in V_j$ such that $v_0\in(v_1,v_2)$, we have that $$\lim{}^\omega \frac{\dv_{1}\bigl(f_n^j(v_1),f_n^j(v_2),f_n^j(v_0);\frac{1}{2ir}\bigr)}{d_n} = \infty$$ by \ref{divcondition} where $r = \d\bigl(f_j(v_0), \{f_j(v_1),f_j(v_2)\}\bigr)$.  Thus $\frac{\dv_{1}\bigl(f_n^j(v_1),f_n^j(v_2),f_n^j(v_0);\frac{1}{ir_n}\bigr)}{d_n} > i $ on an $\omega$-large set where $r_n =  \d\bigl(f_n^j(v_0), \{f_n^j(v_1),f_n^j(v_2)\}\bigr)$.  Since $V_j$ is finite, $B_i$ is the finite intersection of $\omega$-large sets which completes the proof of the claim.

Let $m_n = \max \{i \ | \ n\in B_i\cap A_i \text{ and } i\leq n \}$, if the intersection is non-empty for some $i\leq n$ and $m_n = 1$ otherwise.

Define  $\tilde t = (t_0)$ and $\tilde f\co \lim^\omega_{\tilde t}  T_i \to \con$ by $\tilde f(t) = \bigl(f_n^{m_n}(t)\bigr)$.

Notice that $\lim^\omega_{\tilde t} T_i = \lim^\omega_{\tilde t} V_i$ and $\tilde f\bigl((t_o)\bigr) = f_i(t_0)$ for all $i$.  As in the proof of \ref{embedd}, $\lim^\omega m_n = \infty$ and $\tilde f$ is a well-defined isometric embedding of $\lim^\omega T_i$ into $\con$.

All that remains is to show that $\tilde f$ satisfies the desired separation condition.  Suppose that $v_0 ,v_1,v_2$ are points on $\lim^\omega_{\tilde t} T_i$ such that $v_1,v_2$ are in different components of $\lim^\omega_{\tilde t} T_i\backslash \{v_0\}$.  Then there exist representatives $(v_n^0), (v_n^1), (v_n^2)$ of $v_1, v_2, v_3$ respectively such that $v_n^1, v_n^2$ are in distinct components of $T_n\backslash \{v_n^0\}$ $\omega$-almost surely.  Thus $$\frac{\dv_{1}\bigl(f_n^j(v_n^1),f_n^j(v_n^1),f_n^j(v_n^0);\frac{1}{m_nr_n}\bigr)}{d_n} >m_n $$
on an $\omega$-large set where $r_n =  \d\bigl(f_n^j(v_0), \{f_n^j(v_1),f_n^j(v_2)\}\bigr)$ and $j\leq m_n$.  \ref{divcondition} implies that $\tilde f(v_1),\tilde f(v_2)$ are in distinct components of $\con\backslash \{\tilde f(v_0)\}$.  This completes the proof.

\end{proof}

\begin{prop}\label{fundcutpoint}\ulabel{fundcutpoint}{Proposition}
Let $G$ be a finitely generated group.  If $\gcon$ has a global cut-point, then $\gcon$ is simply connected or has uncountable fundamental group.
\end{prop}

\begin{proof}

We may assume that $G$ is not virtually cyclic, since the theorem is trivial in that case.  Then $G$ has an asymptotic cone $\gcon$ with a global cut-point and more than two ends.  By \ref{Rtree}, $\gcon$ contains an isometrically embedded universal $\mathbb R$-tree $T$ such that the components of $T\backslash\{v\}$ map to distinct components of $\gcon\backslash \bigl\{f(v)\bigr\}$ where $f$ is the isometric embedding of $T$ into $\gcon$.

Suppose that $\gamma\co S^1 \to \gcon$ is an essential loop and fix $x_0\in f(T)$ which we may assume is a base point of $\gamma$.  Let $\rho= 2\diam{\gamma}$ and $S = \{x\in f(T) \ |\ \d(x, x_0) = \rho \}$.  Then $S$ has cardinality continuum and $\d(x,y) =2\rho$ for all $x,y \in S$.  For $x\in S$, choose $g_x\in \prod G$ such that $g_x\cdot x_0 = x$.  Let $S_\gamma = \{g_x\cdot \gamma \ | \ x\in S\}$ which is an uncountable set of essential loops in $\gcon$.

\textbf{Claim:}\emph{ No two loops from $S_\gamma$ are homotopic.}

Suppose that $g_x\cdot\gamma $ is homotopic $g_y\cdot\gamma$.  Then there exists a continuous map $h\co A \to \gcon$ of a planar annulus which takes one boundary component to $g_x\cdot\gamma $ and the other to $g_y\cdot\gamma$.  Since $\d(g_x\cdot\gamma, x_0) >0$, $g_x\cdot\gamma $ and $g_y\cdot\gamma$ are in distinct components of $\gcon\backslash \{x_0\}$.  Thus $h^{-1}\bigr(\{x_0\}\bigl)$ separates the two boundary components of the annulus $A$.  Then there exists a single component $C$ of $h^{-1}\bigr(\{x_0\}\bigl)$ which separates the boundary components of $A$.  This is a consequence of the Phragm\'{e}n-Bower properties (see \cite{hw}).  We can then modify $h$ by mapping the component of the plane bounded by $C$ to $x_0$. This is a null homotopy of $g_x\cdot\gamma $ which contradicts our choice of $\gamma$ and completes the proof of the claim and theorem.

\end{proof}

\begin{cor}\label{fundcutpoint2}\ulabel{fundcutpoint2}{Corollary}
Let $G$ be a finitely generated group.  If $\gcon$ has a global cut-point, then $\gcon$ is simply connected or its fundamental group contains an uncountably generated free subgroup.
\end{cor}

\begin{proof}
Suppose that we have constructed $f\co T\to \gcon$, $\gamma$, $S = \{x\in f(T) \ |\ \d(x, x_0) = \rho \}$, and $S_\gamma = \{g_x\cdot \gamma \ | \ x\in S\}$ as in the proof of \ref{fundcutpoint}.  Let $p_x\co[0,1]\to f(T)$ be the unique geodesic in $f(T)$ from $x_0$ to $x\in S$.

Then $S_\gamma' = \{\mathrm{x} = p_x*g_x\cdot \gamma*\overline p_x \ | \ x\in S\}$ is a set of loops based at $x_0$ $\bigl($ where $\overline p_x (t) = p_x(1-t)\bigr)$.

\textbf{Claim:}\emph{ $S_\gamma'$ generates a free product of cyclic groups.}

Suppose that $\mathrm{x}_1^{n_1}* \cdots*\mathrm{x}_k^{n_k}$ is a null homotopic loop in $\gcon$ where $\mathrm{x}_i \neq \mathrm{x}_{i+1}$, $\mathrm{x}_1 \neq \mathrm{x}_k$ and $\mathrm{x}_i^{n_i}$ is an essential loop.  Then there exists $h\co \mathbb D \to \gcon$ a map from the unit disc in the plane such that $h(\partial \mathbb D)$ is a parameterization of the curve $\mathrm{x}_1^{n_1}* \cdots*\mathrm{x}_k^{n_k}$.  Let $C$ be the closure of the connected component of $h^{-1}\bigl(\gcon\backslash\{x_0\}\bigr)$ containing the subpath $p$ of  $\partial \mathbb D^2$ which maps to  $\mathrm{x}_1^{n_1}$.  By construction, $\partial \mathbb D\cap C = p$ and $h(\partial C\backslash\{p\}) = x_0$.  Define $h'\co \mathbb D \to \gcon$  by $h'(y) = h(y)$ for $y\in C$ and $h'(y) = x_0$ for $y\not\in C$.  Then $h'$ is continuous and $\mathrm{x}_1^{n_1}$ is null homotopic which contradicts our choice of $x_1^{n_1}$.  This completes the proof of the claim.

While the subgroup generated by $S_\gamma'$ may not by a free group  ($\gamma$ might have finite order in the fundamental group), it is the free product of cyclic groups.  Thus it is easy to find an uncountably generated free subgroup.

\end{proof}

\begin{cor}\label{notwide}\ulabel{notwide}{Corollary}
Let $G$ be a finitely generated group.  If $G$ is constricted, then every asymptotic cone of $G$ is simply connected or has uncountable fundamental group.  If $G$ is not wide, then $G$ has an asymptotic cone which is simply connected or has uncountable fundamental group.
\end{cor}

\section{Groups with quasi-isometrically embedded subgroups}\label{section prairie}

\begin{defn}
A group is a \emph{prairie group} if all of its asymptotic cones are simply connected.
\end{defn}

\begin{lem}

The following groups are prairie groups.

\begin{enumerate}
    \item Nilpotent groups;
    \item Hyperbolic groups; and
    \item Groups with quadratic Dehn functions
        \begin{enumerate}
            \item $SL_n(\mathbb Z)$ for $n\geq 5$,
            \item Thompsons group $F$,
            \item Mapping class groups,
            \item $CAT(0)$ groups,
            \item Automatic groups,
            \item Baumslag-Solitar groups $BS_{pp}$, and many many others.
        \end{enumerate}

\end{enumerate}

\end{lem}
\begin{proof}
In \cite{Point}, Point shows that nilpotent groups have a unique asymptotic cone which is homeomorphic to $\mathbb R^n$ for some $n$.  Gromov showed that non-elementary hyperbolic groups have cones which are isometric to a universal $\mathbb R$-tree with uncountable branching at every point.  Papasolgu in \cite{pap} showed that if a group has a quadratic Dehn function then all of its asymptotic cones are simply connected.
\end{proof}

\begin{rmk}
In \cite{ConnerKent}, the author with Greg Conner note that such groups are uniformly locally simply connected; specifically, every loop of length $r$ bounds a disc of diameter at most $Kr$ where $K$ only depends on the group.  However, the discs are not necessarily Lipschitz.
\end{rmk}

\begin{lem}
There exists a finitely presented prairie group such that all of its asymptotic cones have uncountable Lipschitz fundamental group.
\end{lem}

\begin{proof}The discrete Heisenberg group $\pres{x,y,z}{z=[x,y], [x,z]= [y,z]=1}$ is a nilpotent group and hence a prairie group.  In fact every asymptotic cone is homeomorphic to $\R^3$.  However, it is shown in \cite[Theorem 4.10]{DHLT} that the Lipschitz fundamental group of the real Heisenberg group isn't countable generated.

\end{proof}

The key to \ref{fundcutpoint} was that the homotopy between the two loops passed through a cut-point so we could ``cut" the homotopy off to build a null homotopy for one of the loops.  We will show that the same idea holds if the separating set is a highly connected set  instead of a point.  To do this we will require the following is well known covering lemma for open sets in the plain.  We provide a proof for completeness and to fix notation.

\begin{lem}\label{opensquares}\ulabel{opensquares}{Lemma}
Every bounded open set $U$ of $\R^2$ is the union of a null sequence of diadic squares with disjoint interiors.  In addition, the squares can be chosen such that if $A_i$ is the union of squares with side length at least $\frac{1}{2^i}$, then $U\backslash A_i \subset \mathcal N_{\frac{\sqrt 2}{2^{i-1}}}(\partial U)$.

\end{lem}

\begin{proof}
Let $Q_i$ be a sequence of partitions of the plane with the Euclidean metric into closed square discs with side length $\frac{1}{2^{i}}$ such that $Q_i$ refines $Q_{i-1}$.  $Q_i$ can be chosen to be the set of squares with vertices $\bigl\{(\frac{j}{2^i},\frac{k}{2^i}),(\frac{j+1}{2^i},\frac{k}{2^i}),(\frac{j+1}{2^i},\frac{k+1}{2^i}),(\frac{j}{2^i},\frac{k+1}{2^i}) \ | \ j,k \in \mathbb Z\bigr\}$.

Let $D_0$ be the maximal subset of $Q_0$ such that $A_0 \subset U$ where $A_0 = \bigcup\limits_{s\in D_0} s$.  Then $U\backslash A_0 \subset \mathcal N_{\frac{\sqrt 2}{2^{-1}}}(\partial U)$.

We will inductively define $D_i$ and $A_i$ as follows.  Let $D_i$ be the maximal subset of $Q_i$ such that $\bigcup\limits_{s\in D_i} s\subset \overline{U\backslash A_{i-1}}$ where $\overline{U\backslash A_{i-1}}$ is the closure of $U\backslash A_{i-1}$.  Let $A_i = \bigl(\bigcup\limits_{s\in D_i}s\bigr)\cup A_{i-1}$.  We immediately have $U\backslash A_i \subset \mathcal N_{\frac{\sqrt 2}{2^{i-1}}}(\partial U)$.  Then $\bigcup\limits_{i=1}^\infty A_i = U$.

\end{proof}

\begin{defn}\label{modulus}\ulabel{modulus}{Definition}
Let $\xi\co\R^+ \to \R^+\cup\{\infty\}$ be a continuous function which vanishes at $0$. Then $\xi$ is a \emph{modulus of continuity} for $g \co (X, \d_X)\to (Y, \d_Y)$, if $\d_Y\bigr(g(x), g(y)\bigl)\leq \xi\bigr(\d_X(x,y)\bigl)$ for all $x,y \in X$.

Let $(X,\d)$ be a path connected metric space and $\zeta\co\mathbb R^+ \to \mathbb R^+\cup \{\infty\}$ be an increasing function.  We will say that $\zeta$ is a \emph{modulus of path-connectivity} for $(X,\d)$; if every pair of points $x,y\in X $ there exists a path $\alpha$ from $x$ to $y$ such that $\diam{\alpha}\leq \zeta\bigl(\d(x,y)\bigr)$.  If $X$ is geodesic than the identity function is a modulus of path-connectivity for $(X,\d)$.
\end{defn}

\begin{rmk}
Let $g \co (X, \d_X)\to (Y, \d_Y)$ be a continuous function on a compact metric space $X$.  Then $\xi(r) = \sup\bigl\{\d_Y\bigl(g(x),g(y)\bigr) \ | \ \d_X(x,y)\leq r \bigr\}$ is a modulus of continuity which is finite for every $r$.  If $\xi'$ is another modulus of continuity for $g$, then $\xi'(r)\geq \xi(r)$.

Let $(X, \d)$ be a path connected space.  Then there exists a modulus of path-connectivity for $X$ which vanishes at $0$ if and only if $X$ is uniformly locally path connected.
\end{rmk}

\begin{lem}\label{cut1}\ulabel{cut1}{Lemma}
Suppose that $X$ is a metric space containing a closed, simply connected, uniformly locally path connected and uniformly locally simply connected subset $E$.  If $h\co A \to X$ is a continuous map from a planar annulus such that $h^{-1}(E)$ separates the boundary components of $A$, then $h$ takes the boundary components of $A$ to null homotopic loops in $X$.
\end{lem}

\begin{proof}
Let $A = \{(x,y)\in \R^2 \ | \ \frac14\leq x^2 + y^2\leq 1\}$ and $\mathbb D$ be the unit disc in the plane.  It is enough to show that the outer boundary of $A$ maps to a null homotopic loop.  Since $h^{-1}(E)$ separates the boundary components of $A$, a component $C$ of $h^{-1}(E)$ separates the boundaries components of $A$.  This follows from the Phragm\'{e}n-Brouwer properties, see\cite{hw}. Let $U$ be the component of $\mathbb D\backslash C$ which contains the circle of radius $\frac12$.  Thus $\partial U\subset A$ and $h(\partial U)\subset E$.  Let $\xi$ be a modulus of continuity for $h$.

We can decompose $U$ as a null sequence of diadic squares with disjoint interiors, as in \ref{opensquares}.  As before, let $A_{i}$ be the union of squares with side length at least $\frac{1}{2^i}$ which are contained in $U$ and $D_i$ the set of squares in $A_i$ of side length $\frac{1}{2^i}$.  Then $\bigcup\limits_{i=1}^\infty D_i$ induces a cellular structure on $U$.  We will use $U^{(i)}$ to denote the $i$-skeleton of this cellular structure on $U$.  Note this implies that a side of a square in $D_i$ is not necessarily an edge but is an edge path.

We will now define a continuous map $g\co\mathbb D\to X$ such that $g|_{\mathbb D\backslash U} = h$.  If the boundary of $U$ is a loop, then this is obvious.  However, the boundary does not have to be a loop.  It can be very complicated (consider the Warsaw circle).

Let $\iota\co U \to \partial U$ be a closest point projection map (which in general will be discontinuous), i.e. any map such that $ \d(x,\iota(x))\leq \d(x,z)$ for all $z'\in \partial U$.  For every $x\in U^{(0)}$, let $g(x) = h(\iota(x))$.

\begin{clm*}
If $x\in U^{(0)}\backslash A_i$ and $y\in\partial U$, then $\d\bigl(g(x), g(y)\bigr)\leq \xi\bigl(\d(x,y)+ \frac{\sqrt{2}}{2^{i-1}}\bigr)$.
\end{clm*}

If $x \in U^{(0)}\backslash A_i$ and $y\in\partial U$, then $\d\bigl(x, \iota(x)\bigr)\leq \frac{\sqrt{2}}{2^{i-1}}$.  Thus $\d\bigl(\iota(x), y \bigr)\leq \d(x,y)+\frac{\sqrt{2}}{2^{i-1}}$ and the claim follows.

We now wish to extend $g$ continuously to $\mathbb D\backslash U\cup U^{(1)}$.  Let $\zeta\co\mathbb R^+ \to R^+\cup\{\infty\}$ be a modulus of path-connectivity of $X$ which vanishes at $0$.  Then there exists a $\eta>0$ such that $\zeta(t)<\infty$ for all $t<\eta$.  Suppose that $e$ is an edge of $U^{(1)}$ with vertices $x,y$ such that $\d\bigl(g(x),g(y)\bigr)<\eta$.  Then there exists a path $\alpha_{x,y}$ in $X$ from $g(x)$ to $g(y)$ such that $\diam{\alpha_{x,y}}\leq \zeta\bigr(\d(g(x),g(y)\bigl)$.  We may extend $g$ by sending $e$ to $\alpha_{x,y}$.  Repeating this for all sufficiently short edges of $U^{(1)}$ and sending the other edges to any path between their end points, we can extend $g$ to $\mathbb D\backslash U\cup U^{(1)}$.

\begin{clm*}
$g\co \mathbb D\backslash U\cup U^{(1)} \to X$ is continuous.
\end{clm*}

Suppose that $x_n$ is a sequence of points in $U^{(1)}$ such that $x_n \to x_0$.  If $x_0\not\in \partial U$, then $x_n$ is eventually contained in $A_i$ for some $i$ and $g(x_n)\to g(x_0)$ by the Pasting Lemma for continuous functions (see \cite{munkres}).

If $x_0 \in \partial U$, then we can choose $x_n'$ such that $x_n, x_n'$ are contained in a single edge of $D^{(1)}$ and $x_n' \in U^{(0)}$.  As well we may assume that, $x_n$ is contained in a sufficiently short edge (so as to assume the length condition holds on the edge).  Since $x_n$ converges to $\partial U$,  for every $i$ there exists an $N_i$ such that $x_n \in  U^{(1)}\backslash A_i$ for all $n> N_i$.  Then $\d(x_n, x_n') \leq \frac{1}{2^{i}}$ for all $n>N_i$.  Thus $\d\bigl(g(x_n), g(x_n')\bigr) \leq \zeta\bigl(\xi(\frac{1}{2^i})\bigr)$ for all $n>N_i$.  As well, $\d(x_n, x_n') \leq \frac{1}{2^{i}}$ for all $n>N_i$ implies that $x_n'$ converges to $x_0$.

Then \begin{align*}\d(g(x_0), g(x_n)) &\leq  \d(g(x_0), g(x_n'))+ \d(g(x_n'), g(x_n)) \\ & \leq \xi\bigl(\d(x_0,x_n')+ {\sqrt{2}}/{2^{i-1}}\bigr)+ \zeta\bigl(\xi({1}/{2^i})\bigr)\end{align*}
for all $n\geq N_i$.  Thus $g|_{\mathbb D\backslash U \cup U^{(1)}}$ is continuous which completes the second claim.

Let $\epsilon_i = \max\limits_{s\in D_i} \{\diam{g(\partial s)} \}$ which is necessarily finite for all $i$.   Since $g|_{\mathbb D\backslash U \cup U^{(1)}}$ is continuous, $\epsilon_i$ converges to $0$.  Since $E$ is simply connected and uniformly locally simply connected, there exists $\delta_i$ such that for every $s\in D_i$ $g(\partial s)$ bounds a disc with diameter at most $\delta_i$ where $\delta_i \to 0$ as $i\to \infty$.

Fix $i>0$ and $s\in D_i$.  Then we can extend $g$ to all of $s$ by extended $g|_{\partial s}$ to a disc with diameter at most $\delta_i$.

By doing this process for all $s\in \bigcup_{i\geq0} D_i$, we can extend $g$ to all of $\mathbb D$.  Repeating the argument from the second claim and using the fact that $\delta_i\to 0$, we can see that this extension is continuous.

\end{proof}

An interesting related proposition is the following van Kampen type result for fundamental groups.
\begin{prop}

Suppose that $X= U\cup V$ is a connected metric space  and $U\cap V$ is a non-empty, closed, simply connected, uniformly locally path connected and uniformly locally simply connected.  Then for $x_0\in U\cap V$, $\pi_1(V, x_0)*\pi_1(U, x_0)$ canonically embeds into $\pi_1(X, x_0)$.

\end{prop}

If both of $U,V$ are not locally simply connected at $x_0$, then the homomorphism will not necessarily be a surjection.  In fact $\pi_1(X, x_0)\backslash \bigl(\pi_1(V, x_0)*\pi_1(U, x_0)\bigr)$ will often be uncountable.

\begin{proof}
Suppose that $f_i \co (I,0,1) \to (V,x_0, x_0)$ and $g_j\co (I,0,1) \to (U, x_0,x_0)$ are essential loops such that the loop $f_1*g_1* \cdots *f_n*g_n$ is null homotopic in $X$.  Let $h\co \mathbb D \to X$ be a null homotopy and $C$ a component of $h^{-1}(U)$ containing the portion of $\partial D$ which maps to $f_1$.  Since $U\cap V$ is path connected and locally path connected, we can define a map $h'\co C\cup \partial \mathbb D \to U$ such that $h'|_C = h$ and $h'(\partial \mathbb D \backslash C ) \subset U\cap V$.  Then as in \ref{cut1} $h'$ can be extended to a null homotopy of $f_1$ which contradicts the assumption that $f_1$ was an essential loop.
\end{proof}

We will use Olshanskiy's definitions from \cite{Obook} for a \emph{0-refinement of a van Kampen diagram}, \emph{0-edges and 0-cells}, a \emph{cancelable pair} in a van Kampen diagram, a \emph{copy} of a cell under 0-refining, and \emph{reduced diagrams}.  Our definitions of $M$-bands, medians, and boundary paths of $M$-bands will follow that of \cite{OSchord}.

\begin{defn}[$M$-bands]\label{M-bands} Let $M\subset S\cup\{1\}$ where $1$ is the empty word is $S\cup S^{-1}$ and $\Delta$ be a van Kampen diagram over $\pres SR$.  An $M$-edge is an edge in $\Delta$ or $\Gamma(G,S)$ labeled by an element of $M$.  An \emph{$M$-band $\mathcal T$} is a sequence of cells $\pi_1,...,\pi_n$ in a van Kampen diagram over $\pres{S}{R}$ such that

\begin{enumerate}[(i)]
    \item every two consecutive cells $\pi_i$ and $\pi_{i+1}$ in this sequence have a common $M$-edge $e_i$ and
    \item every cell $\pi_i$, $i=1,...,n$ has exactly two $M$-edges, $e_{i-1}$ and $e_i$.
\end{enumerate}

\end{defn}

Consider lines $l(\pi_i, e_i)$ and $l(\pi_i, e_{i-1})$ connecting a point inside the cell $\pi_i$ with midpoints of the $M$-edges of $\pi_i$. The broken line formed by the lines $l(\pi_1,e)$, $\cdots$, $l(\pi_i,e_i)$, $l(\pi_i,e_{i-1})$, $\cdots$, $l(\pi_n,e_n)$ is called the \emph{median} of the band $\mathcal T$ and will be denoted by $ m( \mathcal T)$. It connects the midpoints of each $M$-edge and lies inside the union of $\pi_i$. We say that an $M$-band is an \emph {$M$-annulus}, if $\pi_1$ and $\pi_n$ share an $M$-edge.  If $\mathcal T$ is an $M$-annulus, then the edges $e_1$ and $e_n$ coincide and $m(\mathcal T)$ is a simple closed curve. An $M$-band $\t$ will be \emph{reduced} if no two consecutive cells are inverse images of each other.

Each cell $\pi_i$ of an $M$-band $\t$ can be viewed as an oriented 4-gon with edges $e_{i-1}$,$p_i$, $e_i$, $q_i$ where $e_{i-1}, e_i$ are $M$-edges of $\pi_i$; $p_i$ begins at the initial vertex of $e_{i-1}$ and ends at the initial vertex of $e_i$; and $q_i$ begins at the terminal vertex of $e_{i-1}$ and ends at the terminal vertex of $e_i$.  Then $p_1p_2\cdots p_n$ and $q_1 q_2\cdots q_n$ edge paths in $\Delta$ which we will refer to as the \emph{combinatorial boundary paths of $\t$} and denote by $\topc{\t}$, $\botc{\t}$ respectively.  However, the combinatorial boundary paths can have backtracking in the diagram.  The \emph{(topological) boundary paths of $\t$} are  subpaths of $\topc{\t}$ and $\botc{\t}$ obtained by removing all maximal subpaths consisting entirely of backtracking and will be denoted by $\top\t$ and $\bot\t$ respectively.  While a topological boundary path has no backtracking, its label is not necessarily freely reduced. It is also possible that one of $\top\t$ and $\bot\t$ is empty.

Let $\mathcal T$ be a $M$-annulus in a circular diagram $\Delta$. $\mathcal T$ is a \emph{minimal $M$-annuli}, if there are no $M$-annuli contained in the bounded component of $\R^2\backslash m(\mathcal T)$ where $\Delta$ is considered as a subset of $\R^2$.  $\mathcal T$ is said to be a \emph{maximal $M$-annulus} in $\Delta$ if it is not contained in the bounded component of $\R^2\backslash m(\mathcal T')$ for any other $M$-annulus $\mathcal T'$ in $\Delta$.  For a more complete description of $M$-bands and their boundaries see \cite{OSchord}.

\begin{defn}
Let $G_e$ be an HNN extension of a group $ \pres {A}{R'}$ with finitely generated associated subgroups.  Then $G_e$ has a presentation
$$\pres {A,t}{R' \cup\{u^{t}_{i} = v_{i} \}_{i=1}^k}$$
where $\{u_{1}, . . . , u_{k}\}, \{ v_{1}, . . . , v_{k}\}$ are generating sets for the associated subgroups $H_e =\langle u_i\rangle, K_e = \langle v_i\rangle$.

Let $G_a$ be an amalgamated product of groups $ \pres {A_1}{R_1}$ and $ \pres {A_2}{R_2}$ along $\phi\co H_1 \to H_2$ where $H_i$ is a finitely generated subgroup of $\pres {A_i}{R_i}$.  Then $G_a$ has a presentation
$$\pres {A_1, A_2}{R_1, R_2 \cup\{u_{i} = \phi(u_{i}) \}_{i=1}^k}$$
where $\{u_{1}, . . . , u_{k}\}$ is a generating set for the associated subgroup $H_1$.
\end{defn}

We will fix the groups $G_e$ and $G_a$ and their presentations for the remainder of Section \ref{section prairie}.

\begin{defn}
Let $H$ be a subgroup of a group $G$ generated by $S$ and $Z, Z'$ be subsets of $\Gamma(G,S)$.  We will say that $Z,Z'$ are \emph{$H$-separated} if there exists $g\in G$ such that $Z, Z'$ are contained in distinct components of $\Gamma(G,S)\backslash gH$ where $gH$ is the set of vertices of $\Gamma(G,S)$ labeled by elements from the coset $gH$.  \end{defn}

\begin{lem}
Let $H$ be a subgroup of a group $G$ generated by $S$. The property of being $H$-separated is invariant under the left action of $G$ on $\Gamma(G,S)$.
\end{lem}

\begin{lem}\label{etranslates}\ulabel{etranslates}{Lemma}
Suppose that $H_e$ or $K_e$ is a proper subgroup of $ \pres {A}{R'}$.  Let $\gamma$ be a loop in $\Gamma(G_e, S_e)$ and $N> \diam {\gamma}$.  Then there exists elements $\{g_1,\cdots, g_N\}$ in $G_e$ such that
\begin{enumerate}[(i)]
    \item $g_i\cdot\gamma, g_j\cdot\gamma$ are $H$-separated  for $H\in \{H_e, K_e\}$ and

    \item $|g_ig_j^{-1}|\geq 2N$ and $|g_i|\leq 4N$ for all $i\neq j$.
\end{enumerate}

\end{lem}

\begin{proof}Without loss of generality, we will assume $K_e$ is a proper subgroup.  Let $\gamma$ and $N$ be as in the statement of the lemma.  Choose $a\in \pres {A}{R'}\backslash K_e$ and let $g_i \equiv t^{N}(ta)^i t^{-N}$.  Notice that $g_i$ has no pinches for any $i\in\Z$ and $g_ig_j^{-1} = g_{i-j}$.  For $i\neq j$, $|g_ig_j^{-1}|$ is at least $2N$ since $t^{N}(ta)^{i-j} t^{-N}$ has no pinches.  Being $K_e$-separated is invariant under the action of $G_e$ on $\Gamma(G_e, S_e)$; hence, it is enough to show that $\gamma$ and $g_i\cdot\gamma$ are $K_e$-separated.

Let $x$ be the vertex of $\Gamma(G,S)$ with label $g_i$ and $x_0$ the vertex with label $1$.

Since $t^{N}(ta)^i t^{-N}$ has no pinches, $g_i$ and $1$ are in different components of $\Gamma(G_e, S_e)\backslash T^{N+1}K_e$ where $1$ is the identity element of $G_e$.  As well,  $\d(g_i,T^{N+1}K_e)\geq N$ and $\d(1, T^{N+1}K_e) \geq N$.  Then $N>\diam{\gamma}$ implies that $\gamma$, $g_i\cdot\gamma$ are in distinct components of $\Gamma(G_e,S_e)\backslash T^{N+1}K_e$.

\end{proof}

An analogous proof gives us the following result for $G_a$ where $g_j = a_1^N (a_1a_2)^j a_1^{-N}$ for $a_i\in A_i\backslash H_i$.

\begin{lem}\label{atranslates}\ulabel{atranslates}{Lemma}
Suppose that $H_i$ is a proper subgroup of $G_i$ for $i= 1,2$.  Let $\gamma$ be a loop in $\Gamma(G_a, S_a)$ and $N> \diam {\gamma}$.  Then there exists elements $\{g_1,\cdots, g_N\}$ in $G_a$ such that

\begin{enumerate}[(i)]
    \item $g_i\cdot\gamma, g_j\cdot\gamma$ are $H_1$-separated  and

    \item $|g_ig_j^{-1}|\geq 2N$ and $|g_i|\leq 4N$ for all $i\neq j$.
\end{enumerate}

\end{lem}

\begin{thm}\label{HNN-Amalg}\ulabel{HNN-Amalg}{Theorem}
Suppose that $G$ is an HNN-extension or amalgamated product where the associated subgroups are proper, quasi-isometrically embedded, prairie groups.  Then every asymptotic cone of $G$ is either simply connected or has uncountable fundamental group.
\end{thm}

\begin{proof}
Let $G\in \{ G_e, G_a\}$ and $S$ be the corresponding generating set for $G$.  Suppose that $\gcon$ is not simply connected.  Then there exists $\gamma$ an essential loop in $\gcon$ and we may choose loops $\gamma_n$ in $\Gamma(G,S)$ such that $\bigl(\gamma_n(t)\bigr) = \gamma(t)$.  Let $c_n = 2\diam{\gamma_n}$.  Let $S_n$ be the set of elements of $G$ given by \ref{etranslates} or \ref{atranslates}.  For every two distinct elements $g_n, h_n$ of $S_n$, $g_n\cdot\gamma_n$ and $h_n\cdot\gamma_n$ are $H$-separated for some quasi-isometrically embedded prairie subgroup $H$ of $G$.

Let $g=(g_n), h= (h_n)\in \prod^\omega S_n$.

\begin{clm*}
Then $g\cdot\gamma$, $h\cdot\gamma$ are well-defined loops in $\gcon$ and $g\cdot\gamma$ is not homotopic to $h\cdot\gamma$ if $g, h$ are distinct elements of $\prod^\omega S_n$.
\end{clm*}

The first assertion follows from the fact that $g_n$ grows big O of the scaling sequence.

Suppose that $g\cdot\gamma$ is homotopic to $h\cdot \gamma$ for distinct $h, g$.  Then \as~ $g_n \neq h_n$  and there exists $k_n$ such that $g_n\cdot\gamma_n$ and $h_n\cdot\gamma_n$ are in distinct components of $\Gamma(G,S)\backslash k_n H$.

Thus $g\cdot\gamma$, $h\cdot\gamma$ are in distinct components of $\gcon \backslash \lim^\omega k_n H$ by \ref{limitseparates2}.  Since $H$ is  quasi-isometrically embedded; $\lim^\omega k_n H$ is bi-Lipschitz to $\hcon$ which is simply connected, uniformly locally simply connected, and geodesic.

Thus $\lim^\omega k_n H$ is simply connected, uniformly locally simply connected, and uniformly locally path connected.  Hence, \ref{cut1} implies that $g\cdot\gamma$ and $h\cdot\gamma$ are null-homotopic which contradicts our choice of $\gamma$.

This completes the proof of the claim.  The theorem follows since $\prod^\omega S_n$ is uncountable.
\end{proof}

\begin{cor}If $G$ is has more than one end, then every asymptotic cone of $G$ is either simply connected or has uncountable fundamental group.

\end{cor}

\begin{proof}
If $G$ has more than one end, then it has a graph of groups decomposition with finite edge groups and hence is an HNN extension or an amalgamated product with finite associated subgroups and finite subgroups are always quasi-isometrically embedded prairie groups.
\end{proof}

This corollary was also shown in \cite{DS} since groups with more than one end are relatively hyperbolic.

A lemma due to Burillo.
\begin{lem}[\cite{Burillo}]\label{distortion}\ulabel{distortion}{Lemma}
If $X$ has is quasi-isometric to a metric space with a log metric then every asymptotic cone of $X$ is totally disconnected.
\end{lem}

\begin{cor}\label{HNN-Amalg2}\ulabel{HNN-Amalg2}{Corollary}
Suppose that $G$ is an HNN-extension or amalgamated product where the associated subgroups are exponentially distorted.  Then every asymptotic cone of $G$ is either simply connected or has uncountable fundamental group.
\end{cor}

\begin{proof}
We will proceed as in the proof of \ref{HNN-Amalg}.  We only need to show how to circumvent the use of \ref{cut1}.

We can construct $S_n$ as before and let $g=(g_n), h= (h_n)$ for $g_n, h_n\in S_n$.

If $g_n \neq h_n$ \as, then there exists $X = \lim^\omega k_n H$ such that  $g\cdot\gamma$, $h\cdot\gamma$ are in distinct components of $\gcon\backslash X$.  Since $H$ is exponentially distorted, it is totally disconnected by \ref{distortion}.

Suppose that $h \co A \to \gcon$ is a homotopy from $g\cdot\gamma$ to $h\cdot\gamma$.  Then there exists a component $C$ of $h^{-1}(X)$ which separates the boundary components of $A$.  Since $X$ is totally disconnected, $h(C)$ must be a point.  Hence $h$ can be modified to a map on the disc by sending the component of the disc bounded by $C$ to $h(C)$.  Thus $g\cdot\gamma$ must be null-homotopic, which contradicts our choice of $\gamma$.

\end{proof}

\begin{cor}
Let $G= \pres{a,t}{(a^p)^t = a^q}$ be the Baumslag-Solitar group where $|p|\neq |q|$.  For every $(\omega, d)$, $\gcon$ has the following properties.

\begin{enumerate}[(i)]

    \item $\gcon$ is not semilocally simply connected.
    \item $\pi_1(\gcon,x_0)$ is not simple.
    \item Every decomposition of $\pi_1(\gcon,x_0)$ into a free product of subgroups has a factor which is a not free and uncountable.
    \item $\pi_1(\gcon,x_0)$ contains an uncountable free subgroup.

\end{enumerate}

\end{cor}

\begin{proof}
Let $G= \pres{a,t}{(a^p)^t = a^q}$ be the Baumslag-Solitar group where $|p|\neq |q|$.  Properties $(i)-(iii)$ are proved in Corollary 3.2 of \cite{ConnerKent}.  So we need only prove $(iv)$.  The proof is an adaptation of the proof of \ref{fundcutpoint2}.

Since $\gcon$ is not semilocally simply connected, it is not simply connected.  Thus is contains an essential loop $\gamma$.  \ref{HNN-Amalg} shows how to find an uncountable set of essential loops all of which are in distinct components of $\gcon\backslash \lim^\omega_e g_n \langle a^q\rangle$ for some choice of $g_n \in G$.

Using this uncountable set of loops, we can find $S'_\gamma$ as in \ref{fundcutpoint2}.  We will now use the notation from \ref{fundcutpoint2} and show how to modify the proof.

Suppose that $\mathrm{x}_1^{n_1}* \cdots*\mathrm{x}_k^{n_k}$ is a null homotopic loop in $\gcon$ where $\mathrm{x}_i \neq \mathrm{x}_{i+1}$, $\mathrm{x}_1 \neq \mathrm{x}_k$ and $\mathrm{x}_i^{n_i}$ is an essential loop.  Then there exists $h\co \mathbb D \to \gcon$ a map from the unit disc in the plane such that $h(\partial \mathbb D)$ is a parameterization of the curve $\mathrm{x}_1^{n_1}* \cdots*\mathrm{x}_k^{n_k}$.  Let $C$ be the closure of the connected component of $\mathbb D \backslash h^{-1}\bigl\{\lim^\omega_e g_n \langle a^q\rangle\bigr\}$ containing the subpath $p$ of  $\partial \mathbb D$ which maps to  $\mathrm{x}_1^{n_1}$.

Recall that $\langle a^q\rangle$ is exponential distorted in $G$.  Thus $\lim^\omega_e g_n \langle a^q\rangle$ is totally disconnected by \ref{distortion}.

Since $C$ is the closure of a component of  $\mathbb D \backslash h^{-1}\Bigl(\bigl\{\lim^\omega_e g_n \langle a^q\rangle\bigr\}\Bigr)$, $\partial C \backslash\{ p\}$ is connected and maps into $\lim^\omega_e g_n \langle a^q\rangle$.  Hence $h(\partial C \backslash \{p\})$ is a point $b$.

Define $h'\co \mathbb D \to \gcon$  by $h'(y) = h(y)$ for $y\in C$ and $h'(y) = b$ for $y\not\in C$.  Then $h'$ is continuous and $\mathrm{x}_1^{n_1}$ is null homotopic which contradicts our choice of $x_1^{n_1}$.

Again, the subgroup generated by $S_\gamma'$ may not by a free group but it is the free product of cyclic groups.  Thus it is easy to find an uncountably generated free subgroup.  This completes the proof of the corollary.

\end{proof}

\subsection{Partitions of van Kampen diagrams}\label{section vankampen}

The following definition of partitions are due to Papasoglu in \cite{pap}.

\textbf{Partitions of the unit disc in the plane:} Let $D$ be the unit disk in $\mathbb R^2$ or the planar annulus $\bigl\{(x,y) | x^2 + y^2 \in [\frac14,1]\bigr\}$. A \emph{partition $P$ of $D$} is a finite collection of closed discs $D_1, \cdots, D_k$ in the plane with pairwise disjoint interiors such that $D = \cup_i D_i$, $\partial D = \partial (D_1\cup \cdots\cup D_k)$, and $D_i\cap D_j = \partial D_i\cap\partial D_j$ when $i\neq j$. A point $p$ on $\partial D_1\cup \cdots\cup\partial D_k$  is called a \emph{vertex of the partition} if for every open set $U$ containing $p$, $U\cap
(\partial D_1\cup \cdots\cup\partial D_k)$ is not homeomorphic to an
interval. An \emph{edge of a partition} is a pair of adjacent vertices of a disc in the partition.  A \emph{piece of a partition} is the set of the vertices of a disc in the partition.  A partition is then a cellular decomposition of the underline space of $P$ where each vertex has degree at least 3; so we will use the standard notation, $P^{(i)}$, to denote the $i$-th skeleton of a partition.

\textbf{Geodesic $n$-gons in a metric space X:} An \emph{$n$-gon} in $X$ is a
map from the set of vertices of the standard regular $n$-gon in the plane
into $X$, i.e. an ordered set of $n$ points in $X$.  If $X$ is a geodesic
metric space, we can extend the $n$-gon to edges by mapping the edge
between adjacent vertices of the standard regular $n$-gon in the plane to a geodesics segment joining the corresponding vertices of the $n$-gon in $X$.  We will say that such an extension is a \emph{geodesic $n$-gon} in $X$.

\textbf{Partitions of loops in a geodesic metric space X:} Let $\mathbb D$ be the unit disc in the plane  and $\gamma\co\partial \mathbb D\to X$ be a continuous map.
A \emph{partition of $\gamma$} is a map $\Pi$ from the set of vertices of a partition $P$  of $\mathbb D$ to $X$ such that $\Pi\bigl|_{\partial P\cap P^{(0)}}=\gamma\bigl|_{\partial P\cap P^{(0)}}$. The \emph{vertices/edges/pieces} of $\Pi$ are the images of vertices/edges/pieces of $P$.  We will write $\Pi(\partial D_i)$ for the pieces of $\Pi$, where $D_i$ are the $2$-cells of the partition $P$.

\begin{rmk}\label{extend}\ulabel{extend}{Remark}
Suppose that $\Pi\co P^{(0)}\to X$ is a partition of a loop $\gamma$ in a geodesic metric space.  We can extend $\Pi$ to $P^{(1)}$ by mapping every edge contained in $\partial P^{(2)}$ to the corresponding subpath of $\gamma$ and every edge not contained in $\partial P^{(2)}$ to a geodesic segment joining its end points.  Then the \emph{length of a piece} is the arc length of the loop $\Pi(\partial D_i)$.  We will write $|\Pi(\partial D_i)|$ for the length of the piece $\Pi(\partial D_i)$. We define the \emph{mesh of $\Pi$} by $$\mesh (\Pi) = \max\limits_{1\leq i\leq k} \{|\Pi(\partial D_i)|\}.$$

At times it will be convenient to ignore some pieces of a partition.  If $Z$ is a subset of the pieces of $P$, then the \emph{relative mesh of $\Pi$} is $$\rmesh_Z (\Pi) = \max\limits_{D_i\in Z} \{|\Pi(\partial D_i)|\}.$$
\end{rmk}

When $X$ is a Cayley graph of a group, we will also assume that the partition takes vertices of $P$ to vertices in the Cayley graph.  A partition $\Pi$ is called a \emph{$\delta$-partition}, if $\mesh \Pi<\delta$.  A loop of length $k$ in a geodesic metric space is \emph{partitionable} if it has a $\frac{k}{2}$-partition.

Let $P(\gamma,\delta)$ be the minimal number of pieces in a $\delta$-partition of $\gamma$ if a $\delta$-partition exist and infinity otherwise.

If $P$ is a partition of the unit disc in the plane, then $P^{(1)}$ can be considered as a planar graph where every vertex has degree at least 3.   Then one can use the fact that the Euler characteristic of a planar graph is 1 to obtain the following.

\begin{lem}
Let $\Pi \co P^{(1)}\to \Gamma(G,S)$ be a partition of a loop in the Cayley graph of $G$.  If $\Pi$ has $F$ pieces, then $\Pi$ has at most $3F$ edges and at most $2F$ vertices.
\end{lem}

A straightforward inductive argument gives us the following lemma.

\begin{lem}\label{degree3vertices}\ulabel{degree3vertices}{Lemma}
Suppose that $T$ is a finite simplicial tree with at most $j$ vertices of degree $1$.  Then $T$ has at most $j-1$ vertices with degree greater than $2$.
\end{lem}

Our goal for the remainder of Section \ref{section vankampen} and Section \ref{section HNN} is to define partitions of van Kampen diagrams and show how to use the standard techniques for reducing van Kampen diagrams to build nice partitions of loops in the Cayley  graph.

\begin{defn}
Suppose that $\pres{A, t}{R}$ is an HNN-extension with stable letter $t$. Let $w$ be a word in the alphabet $S\cup S^{-1}$.  We will use $|w|_F$ to denote the freely reduced word length of $w$, $|w|_G$ to denote the minimal word length of $w$ in $\pres{A,t}{R}$ and $|w|_t$ to denote the number of $t$-letters in $w$.

A word $w$ is a \emph{$t$-shortest word} if $|w|_t\leq |w'|_t$ for all $w'=_G w$ and \begin{equation}\label{equation1}|w|_G = |w|_t +\sum |v_i|_G\end{equation} where $v_i$ ranges over maximal $a$-subwords of $w$.  To avoid trivialities, we will also require that every $a$-subword of a $t$-shortest word be freely reduced.

We will say that $w$ is an \emph{almost $t$-shortest word} if $|w|_t\leq |w'|_t$ for all $w'=_G w$.

A path $\gamma$ in the Cayley graph of $G$ is a \emph{$t$-shortest path} (or  an \emph{almost $t$-shortest path}) if $\lab\gamma$ is a $t$-shortest word (or an almost $t$-shortest word).
\end{defn}

The equality in (\ref{equation1}) implies that if we replace each maximal $a$-subword of a $t$-shortest path with a geodesic, then the whole path is geodesic.  This gives us the following result.

\begin{lem}\label{almostgeodesic}\ulabel{almostgeodesic}{Lemma}
Every edge in $\Gamma (G,S)$ labeled by a $t$-letter on a $t$-shortest path from $g$ to $h$ is also an edge of a geodesic from $g$ to $h$.
\end{lem}

\begin{defn}
Let $P$ be a partition of the unit disc $\mathbb D^2$ or the unit annulus in the plane and $\Delta$ a van Kampen diagram over $\pres {S}{R}$. A continuous map $\Psi\co P^{(2)}\to\Delta $ is a \emph{partition of $\Delta$} if it satisfies the following conditions.

\begin{enumerate}[(i)]

    \item $\Psi(P^{(0)})\subset \Delta^{(0)}$

    \item $\Psi$ takes edges of $P$ to edge paths in $\Delta^{(1)}$

    \item For each closed 2-cell $D$ of $P$, $\Psi(D)$ is a reduced subdiagram of $\Delta$

\end{enumerate}

If we consider $\Delta$ as a metric space with the edge metric, then $\Psi|_{P^{(0)}}$ is a partition of the loop $\partial\Delta$ under our previous definition.

As before, the \emph{edges/vertices/pieces of $\Psi$} are the image under $\Psi$ of edges/vertices/pieces of $P$ in $\Delta$.

Define the \emph{mesh of $\Psi$} by $\mesh(\Psi)=\mesh(\theta\comp\Psi)$ where $\theta$ is the canonical map into the Cayley complex.

$\Psi$ is an \emph{$h$-partition} of $\Delta$, if $\Psi$ is  partition of $\Delta$ and a homeomorphism.  
If $\Psi$ is a $h$-partition of $\Delta$ and  $\theta\comp\Psi$ takes edges of $P$ to geodesic paths ($t$-shortest paths), then we will say $\Psi$ is a \emph{geodesic partition ($t$-shortest partition)} of $\Delta$.

\end{defn}

This gives the underling space of $\Delta$ two cell structures, the cell structure inherited as a van Kampen diagram and the cell structure inherited from the partition.  When there is a chance of confusion, we will specify if we are considering a vertex/edge in the underling space as a $\Psi$-vertex/$\Psi$-edge or a $\Delta$-vertex/$\Delta$-edge.

The following lemma follows trivially by considering each of the three types of $0$-refinements.

\begin{lem}\label{adapt}\ulabel{adapt}{Lemma}
    Suppose that $\Psi\co P^{(2)}\to \Delta$ is a partition (or a geodesic partition) and $\Delta'$ is a $0$-refinement of $\Delta$. Then there exists a partition (or geodesic partition) $\Psi\co P^{(2)}\to \Delta'$ which preserves the number of pieces, edges, and vertices; the mesh of the partition; and the labels of edges (after removing any possible $1$'s).
\end{lem}

\section{HNN extensions with free associated subgroups}\label{section HNN}

Let $G$ be a multiple HNN extension of a free group $F$ with free associated subgroups.  Then $G$ has a presentation
$$\pres{ A\cup\{t_i\} }{ \{u^{t_i}_{i,s} = v_{i,s} \} \text{ for } \ i = 1, . . . , k \text{ and } s = 1, . . . , j_i }$$
where $U_i = \langle u_{i,1}, . . . , u_{i,j_i}\rangle, V_i = \langle v_{i,1}, . . . , v_{i,j_i}\rangle$ are free subgroups with free generating sets $\{u_{i,j}\}$, $\{v_{i,j}\}$ respectively and $t_i$ are stable letters.  We will use $\pres{S}{R}$ to denote this presentation for $G$ which we will fix  throughout Section \ref{section HNN}.  Let $$K = \max\{|u_{i,1}|_F, . . . , |u_{i,j_i}|_F, |v_{i,1|_F}, . . . , |v_{i,j_i}|_F\}.$$
We will also fix the constant $K$ throughout this section.  To simplify notation, we will frequently refer to $t_i$-bands in diagrams over $\pres{S}{R}$ as just $t$-bands when the specific $i$ is inconsequential.

\begin{lem}\label{cutlemma}\ulabel{cutlemma}{Lemma}
Let $\t$ be a $t$-band in a van Kampen diagram $\Delta$. Then $\Delta$ can be  modified while preserving the numbers of cells and the boundary label of $\Delta$ such that the label of $\top\t$ and $\bot\t$ are freely reduced words.
\end{lem}

\begin{proof}
If $\lab{\bot\t} = w_1 u u^{-1} w_2$, then we may cut $\Delta$ along the subpath of $\bot\t$ labeled by $u u^{-1}$ and re-identify them as in Figure \ref{diamondmove}.  This is the so called \emph{diamond move} (see \cite{CollinsHuebschmann}).  A similar process can be performed for $\top\t$.
\end{proof}

\begin{figure}[h]
\centering
\def\svgwidth{\columnwidth}
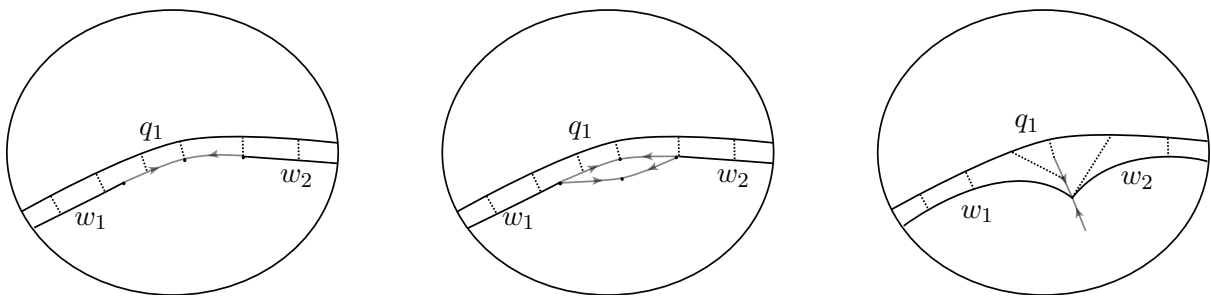\caption{Modifying $\Delta$ to insure that the label of the boundary of a $t$-band is freely reduced.}\label{diamondmove}
\end{figure}

\begin{lem}\label{correct}\ulabel{correct}{Lemma}
Suppose that $\t$ is a reduced $t$-band in a van Kampen diagram $\Delta$ over $\pres{S}{R}$ endowed with the edge metric.  Then there exist an $L$ such that $\top\t$ is in the $L$-neighborhood of $\bot\t$ where $L$ is a constant depending only on the associated subgroups.
\end{lem}

\begin{proof}
The lemma is trivial if you are considering $\topc\t$ and $\botc\t$ in placy of $\top\t$ and $\bot\t$. So we will prove the lemma by finding a bound on the diameter of the backtracking that was removed to obtain $\top\t$.

Recall that $\{u_{i,1}, . . . , u_{i,j_i}\}, \{ v_{i,1}, . . . , v_{i,j_i}\}$ are free generating sets for the associated subgroups where $u_{i,k},v_{i,k}$ are words in the alphabet $A$.   For the purposes of this lemma; let $U$ be the disjoint union of $ \langle u_{i,1}, . . . , u_{i,j_i}\rangle$ and $\langle v_{i,1}, . . . , v_{i,j_i} \rangle $  and if $g\in U$, let $|g|_s$ denote the length in the associated subgroup.  Let

$$L' = \max\{ |g|_s \ | \ g\in U \text{ and } |g|_G\leq 2K \}.$$

Fix $\t$ a reduced $t$-band in $\Delta$ and $v$ a vertex on $\top\t$.  Then there exists a vertex $v'$ on $\botc\t$ such that $\d(v,v')\leq K+1$.   Suppose that $p$ is a maximal subpath of $\botc\t$ which contains the vertex $v'$ and has freely trivial label in $F(A)$.  We will assume (without loss of generality) that $\botc\t$ is labeled by words from $\{u_{i,1}, . . . , u_{i,j_i}\}$.  Then for some $j$, $\lab p = w_1 u_{j,s_1}^{\epsilon_{1}}\cdots u_{j,s_r}^{\epsilon_{r}}w_2$ where $w_1$ is a terminal segment of $u_{j,s_0}^{\epsilon_0}$, $w_2$ is an initial segment of $u_{j,s_{r+1}}^{\epsilon_{r+1}}$, and $\epsilon_{i}= \pm1$.  Let $g= u_{j,s_0}^{\epsilon_{0}}\cdots u_{j,s_{r+1}}^{\epsilon_{r+1}}$.  By construction $|g|_G$ is at most $2K$ and in $U$.  Thus $|g|_s \leq L'$.  This implies that $v'$ is at most $L'K$ from a vertex of $\bot\t$.

Thus $v$ is at most $L= L'K + K + 1$ from a vertex of $\bot\t$ which completes the lemma.

\end{proof}

The following lemma is a correction of a lemma by Olshanskii and Sapir in \cite{OS1}.

\begin{lem}\label{diamater}\ulabel{diameter}{Lemma}There exists a constant $L$ such that every diagram over $\pres{S}{R}$ which has no $t$-annuli and all $t$-bands are reduced has diameter no greater than $\frac{3L|\partial\Delta|}{2}$.
\end{lem}

\begin{proof}  Let $L$ be the constant from \ref{correct}.

Let $s$ be the number of $t$-bands in $\Delta$ and $n=|\partial\Delta|$. Then $s\leq\frac n2$. There exists a $t$-band $\t$ such that (without loss of generality) $\topc\t$ is contained in $\partial\Delta$ (see Lemma 2.1 of \cite{OS1}). Then $\Delta$ is obtained by gluing $\t$ and a diagram $\Delta_1$ with $s-1$ $t$-bands which satisfies the same hypothesis. Every vertex on a $\bot\t$ can be connected to the boundary of $\Delta$ by a path of length at most $L$. By induction on $s$, we can deduce that every vertex inside $\Delta$ can be connected to the boundary of $\Delta$ by a path of length at most $Ls\leq \frac{Ln}{2}$.
Hence the diameter of $\Delta$ is at most $\frac{3Ln}{2}$.
\end{proof}

\begin{lem}\label{diameter2}\ulabel{diameter2}{Lemma}
Let $\Delta$ be a van Kampen diagram with no $t$-annuli, every $t$-band reduced, and $\gamma\co[0,1]\to \Delta$ be a parametrization of $\partial \Delta$. Suppose that $0= t_0< t_1<\cdots< t_k<t_{k+1}=1$ is a partition of the unit interval and $I$ a subset of $\{0,\cdots, k\}$ such that $\gamma$ restricted to $[t_i,t_{i+1}]$ is a $t$-shortest path for $i \in I$.  Then $\theta(\Delta)$ has diameter no greater than
$$\frac{5L}{2}\Bigl(\sum_{i\not\in I} \Bigl|\gamma|_{[t_i,t_{i+1}]}\Bigr| + \sum_{i\in I} \d\bigl(\theta\comp\gamma(t_i),\theta\comp\gamma(t_{i+1})\bigr)\Bigr)$$
where $L$ is the constant for \ref{correct} and $\theta$ is the canonical map into the Cayley graph.
\end{lem}

\begin{proof}

Let $C = \sum_{i\not\in I} \bigl|\gamma|_{[t_i,t_{i+1}]}\bigr| + \sum_{i\in I} \d\bigl(\theta\comp\gamma(t_i),\theta\comp\gamma(t_{i+1})\bigr)$.
By the same argument as in \ref{diameter}, every vertex of $\Delta$ can be connected to a vertex on $\partial\Delta$ by a path of length at most $Ls$ where $s$ is the number of $t$-bands in $\Delta$.

For $i\in I$, let $w_i = \lab{\gamma|_{[t_i,t_{i+1}]}}$ and $\tilde w_i$ be a geodesic word obtained by replacing each maximal $a$-subpath of $w_i$ by a geodesic word.  For $i\not\in I$, let $w_i = \lab{\gamma|_{[t_i,t_{i+1}]}} = \tilde w_i$.  Then $C = |\tilde w_0\tilde w_1\cdots \tilde w_k|$.  Fix $\Delta_i$ a reduced van Kampen diagram with $\partial\Delta_i = p_i\tilde p_i$ where $\lab{p_i} = w_i$ and $\lab{\tilde p_i} = \tilde w_i^{-1}$.  Let $s_i$ be the number of $t$-bands in $\Delta_i$.  Since no $t$-band of $\Delta_i$ can start and stop on $p_i$; hence, $s,s_i \leq \frac C2$. By repeating the arguments from \ref{diameter}, we can see that any point in $\Delta_i$ is at most $Ls_i$ from a point on $\tilde p_i$.  Hence, if $x,y$ are two points on $\partial \Delta$, then $\d\bigl(\theta(x), \theta(y)\bigr)\leq Ls_i + Ls_j+\frac C2\leq \frac{LC}2 + \frac{LC}2 + \frac C2$.

If $x,y$ are two points in $\Delta$; $\d\bigl(\theta(x), \theta(y)\bigr)\leq 2(Ls) + (Ls_i + Ls_j+\frac C2)$.  Therefore $\theta(\Delta)\leq LC + (LC + \frac C2)\leq \frac{5LC}2$.

\end{proof}

\begin{rmk}\label{induce}
Let $\Pi \co P^{(0)}\to \Gamma(G,S)$ be a partition of a loop $\gamma$ in $\Gamma(G,S)$.  We can extend $\Pi$ to $P^{(1)}$ as in \ref{extend}; but instead of mapping the interior edges of $P$ to geodesics, we will map the interior edges to $t$-shortest paths in $\Gamma(G,S)$.  We can label the edges of $P^{(1)}$ with the label of their image.  Then we can fill each piece with a reduced circular van Kampen diagram.  This produces a van Kampen diagram with boundary label equal to the $\lab{\gamma}$ and $\Pi$ induces a canonical homeomorphism from $P^{(2)}$ onto this van Kampen diagram.  Thus every partition $\Pi$ of $\gamma$ induces a $t$-shortest partition $\Psi$ of a diagram such that  $\Pi = \theta\circ \Psi$.  Then by \ref{diameter2}, each subdiagram corresponding to a piece has diameter at most $\frac{5Ln}{2}$.
\end{rmk}

\subsection{Removing $t$-bands from partitions}\label{subsection removing_t-bands}

\begin{defn}  Suppose $\Psi\co P^{(2)} \to \Delta$ is a $t$-shortest partition of a van Kampen diagram $\Delta$.
A $t$-band $\t$ \emph{crosses} a $\Psi$-edge $e$, if $e$ contains a $t$-edge from $\t$.
If $\t$ is a $t$-annulus which crosses a $\Psi$-edge $e$, we will call the end points of the corresponding $t$-edge, the \emph{crossing vertices} of $\t$.
\end{defn}

\begin{lem}\label{crosses}\ulabel{crosses}{Lemma} If $\Psi$ is a $t$-shortest partition of $\Delta$ and $\t$ is a $t$-band in $\Delta$, then $\t$ crosses each $\Psi$-edge at most once.
\end{lem}

\begin{figure}[t]
\centering
\def\svgwidth{3in}
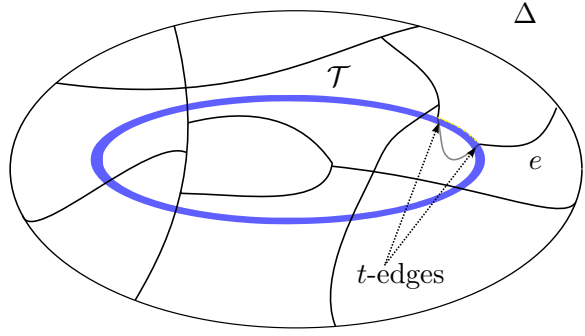\caption{A $\Psi$-edge which crosses $\t$ twice cannot be $t$-shortest.}\label{cross}
\end{figure}

\begin{proof}
If $\t$ crossed a $\Psi$-edge $e$ twice, then $e$ would contain two $t$-edges and the subword of $e$ beginning and ending with these $t$-edges would be equal to a subword of $\topc\t$ or $\botc\t$.   See Figure \ref{cross}.  Thus $e$ was not $t$-shortest. (Note we are using the fact the diagrams are planar.)
\end{proof}

\begin{cor}
    Let $\Psi\co P^{(2)} \to \Delta$ be a $t$-shortest partition of $\Delta$ and $\t$ be a $t$-annulus in $\Delta$.  Then the bounded component of $P^{(2)}\backslash\Psi^{-1}(m(\t))$ contains a vertex of $P$.
\end{cor}

\begin{cor}\label{containsvertex}\ulabel{containsvertex}{Corollary}
    Let $\Psi\co P^{(2)} \to \Delta$ be a $t$-shortest partition of $\Delta$.  Then $\Delta$ can have at most $V$ maximal $\t$-annuli where $V$ is the number of vertices of the partition $\Psi$.
\end{cor}

\begin{lem}\label{removing}\ulabel{removing}{Lemma}
Let $\Psi\co P^{(2)}\to \Delta$ be a $h$-partition of $\Delta$ with $F$ pieces where $\Delta$ is an annular diagram where the boundary components have labels which are trivial in $G$.  Suppose that $\t$ is a $t$-annulus in $\Delta$ such that $\t$ crosses each edge at most once and if $v$ is a crossing vertex of a $\Psi$-edge with vertices $e_-,e_+$, then $\d(e_-, e_+)\leq \d(e_-, v) + \d(v, e_+)$.  Let $B= \max\limits_{D\in P} \{\diam{\theta\comp\Psi(D)}\}$.

Then there exists a partition $\widetilde\Psi\co\widetilde P^{(2)}\to \Delta'$ where $\Delta'$ is obtained by removing $\t$ such that

\begin{enumerate}[(i)]
    \item $\widetilde\Psi$ has no more than $9F^2+4F$ pieces, and
    \item $\mesh{(\widetilde\Psi)}\leq \max\{3(B+2K), \mesh{(\Psi)}\}$,

\end{enumerate}
 where $K$ is the max of the word length of the generators of the associated subgroups.
\end{lem}

\begin{proof}
Let $\Psi\co P^{(2)}\to \Delta$ be a partition of $\Delta$ as in the statement of the lemma and let $A$ be the underline space of $P$.  Let $\Delta_A$ be the subdiagram of $\Delta$ obtained by removing all cells interior to $\topc\t$.

Let $V= \{v_1, v_2, \cdots, v_k\}$ be the set of crossing vertices of $\t$ which are contained in $\topc\t$  where the ordering is obtained by traversing $\topc\t$ in the clockwise direction.  Let $q_i$ be a subpath of $\topc\t$ between $v_i$ and $v_{i+1}$ without backtracking (where the indices are taken modulo $k$) which intersects $V$ only at $v_i, v_{i+1}$ and $m(q_i)$ the corresponding subpath of $m(\t)$.  Since $\t$ crosses each vertex at most once, $k\leq 3F$.

By construction $m(q_i)$ is contained inside of $\Psi(D)$ for some piece $D$ of $P$.  Thus $q_i$ is in the $K$-neighborhood of $\Psi(D)$ and $\diam{\theta\comp\Psi(q_i)}\leq B+2K$

\setcounter{clm}{0}
\begin{clm}  There exists a refinement $P'$ of $P$  and a partition $\Psi'\co P'^{(2)}\to \Delta$ with $\Psi'(x) =\Psi(x)$  for all $x\in P^{(1)}$ such that

\begin{enumerate}[(i)]
    \item the number of pieces of $P'$ is less than $4F$;
    \item $\mesh{(\Psi')}\leq\mesh{(\Psi)}$; and
    \item there is a simple closed curve $\beta_\t$ in $P'^{(1)}$ such that
        \begin{enumerate}
            \item $\Psi(\beta_\t)\subset \topc\t$,
            \item $\beta_\t$ has at most $3F$ edges, and
            \item if $\Psi'(x)$ is interior to $m(\t)$, then $x$ is interior to $\beta_\t$.
        \end{enumerate}

\end{enumerate}

\end{clm}

\begin{figure}
\centering
\def\svgwidth{\columnwidth}
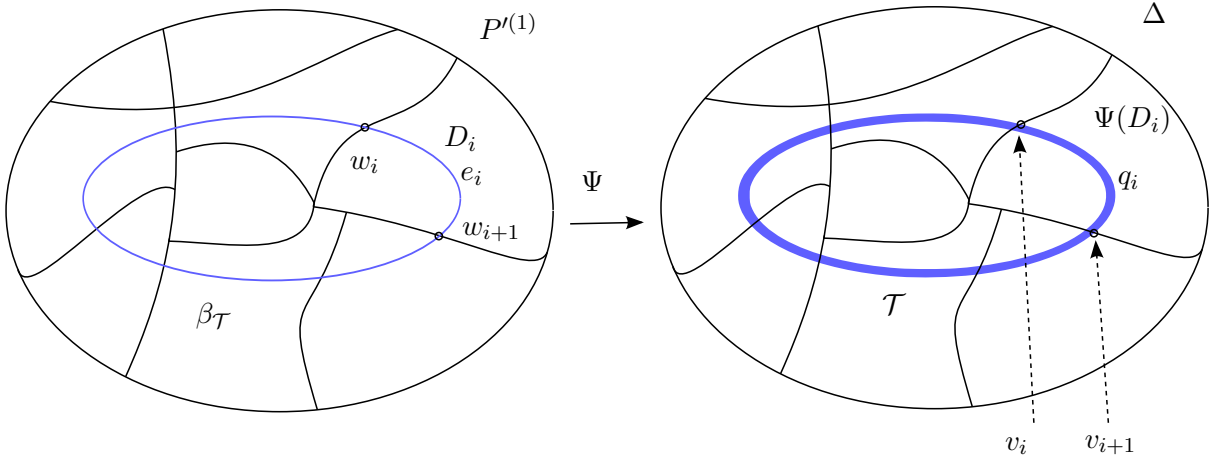\caption{Constructing $P'$}\label{fig1}
\end{figure}

\begin{proof}[Proof of Claim 1]

Let $w_i = \Psi^{-1}(v_i)$ and $W = \{w_i\}$.  For each pair $i$, there exist a unique cell $D_i$ of $P$ such that $\Psi^{-1}\bigl(m(q_i)\bigr)\subset D_i$.  Let $e_i$ be an arc in $D_i$ from $w_i$ to $w_{i+1}$ such that $e_i\cap P^{(1)} = \{w_i, w_{i+1}\}$.  In addition, we may assume that the arcs $e_i$ have disjoint interiors. Then $\beta_\t = e_1*e_2* \cdots *e_k$ is a simple closed curve.

Let  $P'^{(0)} = P^{(0)}\cup W$.  The edges of $P'$ are the closure of the connected subsets of $P^{(1)}\cup \beta_\t \backslash P'^{(0)}$. This gives $\beta_\t$ a cellular structure.  Each vertex of $\beta_\t$ corresponds to a crossing vertex of $\partial_o \t$.  Since $\beta_\t$ has at most $3F$ vertices and each edge cuts a piece of $P$ into two pieces, $P'$ has at most $4F$ pieces.

We can define $\Psi'|_{P^{(1)}} = \Psi$ and map $e_i$ to  $q_i$.  By \ref{cutlemma}, we may also assume that $\Psi'(e_i)$ has freely reduced label.  We can extend $\Psi'$ to the $2$-cells of $P'$ in the natural way.  Then $\Psi'\co P'^{(2)}\to\Delta$ is a partition of $\Delta$ which satisfies the first and third conditions of the claim.

The geodesic condition on crossing vertices guarantees that the mesh does not increase as we add the vertices $v_i$ and the edges $e_i$.

\end{proof}

Claim 1 gives us that $\Psi(\beta_\t)$ bounds a subdiagram of $\Delta$ with freely trivial boundary label and $\beta_\t$ bounds a subcomplex of $P'^{(1)}$.  There exist a simplicial tree $L_\t$ labeled by $a$-letters and a map $\Upsilon \co \beta_\t \to L_\t$ such that $\theta\circ\Psi'|_{\beta_\t} = \theta'\circ\Upsilon$ where $\theta'$ is a label preserving map from $L_\t$ into $\Gamma(G,S)$.  $L_\t$ is constructed by choosing a free reduction of $\lab{\Psi'(\beta_\t)}$.

We can replace the subdiagram in $\Delta$ bounded by $\Psi(\beta_\t)$ with $L_\t$.  This creates a pairing of $\Delta$-edges in $\Delta$.  What we want to be able to do is mirror this identification of edges on $\beta_\t$.  The problem is that this identification can pair proper segments of edges in $\beta_\t$.  To correct this we will need to add new vertices to $P'$ to insure that this identification respects $\Psi'$-edges.  In general, this will cause the mesh to increase since edges of $\beta_\t$ do not map to geodesics. So we will subdivide pieces to get a useful bound on our new mesh.  This is where the bound $B$ on the diameter of each piece comes into play.

We will say that a subpath of $\beta_\t$ is an \emph{$L_\t$-segment}, if all vertices of the edge path except possible the initial and terminal vertices have degree 2 in $\Upsilon(\beta_\t)$.

\begin{clm}

There exists a refinement $P''$ of $P'$ and a partition $\Psi''\co P''^{(2)}\to \Delta$ with $\Psi''(x) =\Psi'(x)$  for all $x\in P''^{(2)}= P'^{(2)}$ such that

\begin{enumerate}[(i)]
    \item the number of pieces of $\Psi''$ is no more than $9F^2+4F$,
    \item $\beta_\t$ is subdivided into at most $9F^2$ edges and each edge is an $L_\t$-segment, and
    \item $\rmesh_Z{(\Psi'')}\leq \max\{3(B+K), \mesh{(\Psi)}\}$ where $Z$ is the set of pieces of $P''$ which are not interior to $\beta_\t$.

\end{enumerate}
\end{clm}

\begin{figure}[ht]
\centering
\def\svgwidth{\columnwidth}
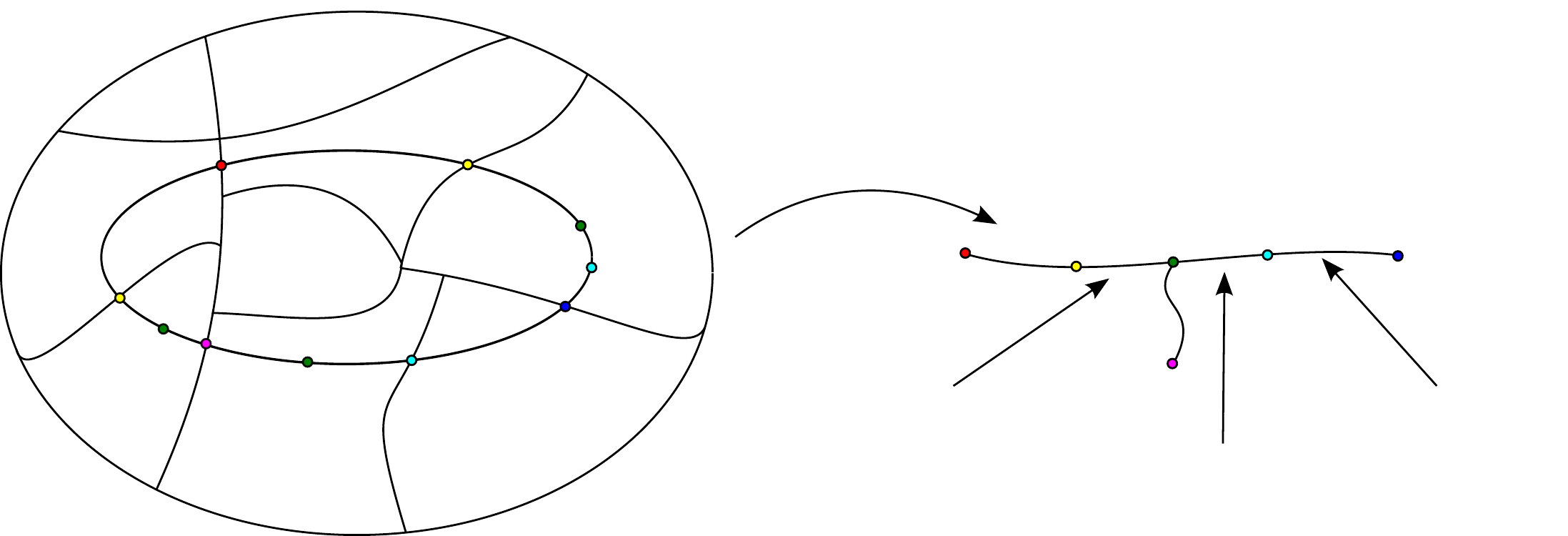\caption{$P''^{(1)}_1$ and $ L_\t$}\label{fig2}
\end{figure}

\begin{proof}[Proof of Claim 2]

$\Upsilon$ must map each $e_i$ injectively into $L_\t$, since $\lab{\Psi'(e_i)}$ is freely reduced.  Thus a vertex of $L_\t$ with degree 1 must be the image of a vertex of $e_i$ for some $i$ and $\Upsilon(\beta_\t)$ has at most $3F$ vertices of degree 1.  Then \ref{degree3vertices} implies that it has at most $3F$ vertices of degree greater than 2.  For each $i$, we can add new vertices to $e_i$ which are the unique $\Upsilon$-preimage of vertices of $L_\t$ with degree greater than 2 or the unique $\Upsilon$-preimage of a point of $\Upsilon(W)$ (see Figure \ref{fig2}).  Doing this subdivides $e_i$ into at most $3F$ edges which we will label by $e^i_j$ with their ordering induced by $e_i$.  This divides $\beta_\t$ into at most $9F^2$ edges.

Let $P_1''$ be the cellular decomposition obtained by adding $\{e_i^j\}$ to $P'$.  Notice the $ P_1''$ is not a partition of $A$ since it has vertices of degree $2$.

In $P'$ there existed exactly two pieces which share $e_i$ as a common edge, $p_i$ which is contained in the bounded component of $\R^2\backslash\beta_\t$ and $p_o$ which is contained in the unbounded component (see Figure \ref{fig2}).

\begin{figure}[ht]
\centering
\def\svgwidth{\columnwidth}
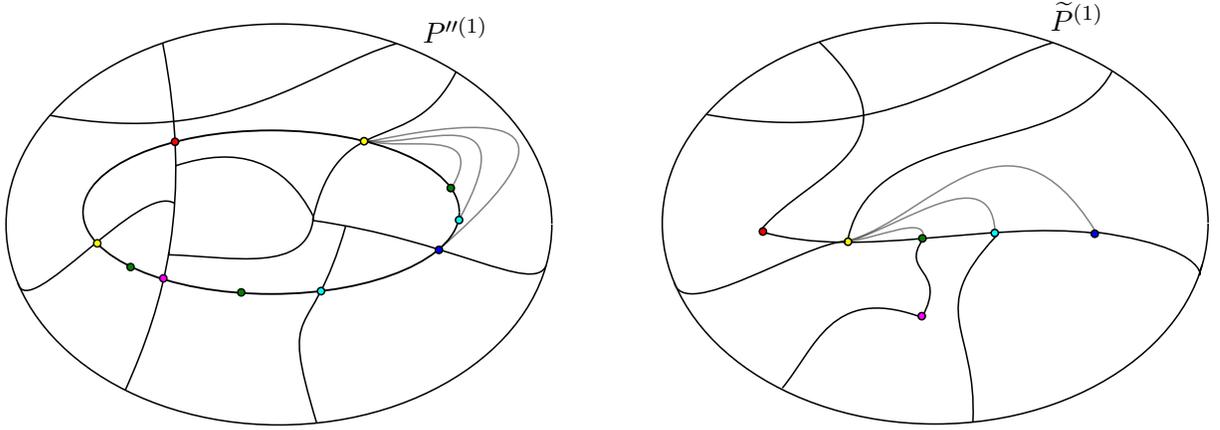\caption{Constructing $\widetilde P$}\label{fig3}
\end{figure}

We will now subdivide the piece $p_o$ to obtain pieces with bounded mesh (see Figure \ref{fig3}).  Let $f_j^i$ be an arc in $p_o$ from the initial vertex of $e^i_1$ to the terminal vertex of $e^i_j$ for all $j>1$.  We also will require that the new edges have disjoint interiors contained in $p_o$.  This subdivides $p_o$ into at most $3F+1$ pieces, i.e. we add $3F$ pieces to our count.  Repeating this process for each $i$, gives us a partition $P''$ of $A$.

We must now explain how to map these edges into $\Delta$. Each new edge connects points with image on the $\Psi(\beta_\t)$.  Thus we can send each edge to the reduced subpath of $\Psi(\beta_\t)$ connecting the images of their vertices and map the 2-cells in the natural way.  Let $\Psi''\co P''^{(2)}\to \Delta$ be this new partition.

The distance between $e_j^i$ and $e_{j'}^i$ is at most $B+2K$ for all $j$ and $j'$. This implies that the requirement on the mesh is then satisfied.

\end{proof}

We can replace the subdiagram of $\Delta$ bounded by $\Psi'(\beta_\t)$ with $L_\t$, creating a new van Kampen diagram $\Delta'$.  This also induces a paring of edges on $\beta_\t$ such that after removing the disc bounded by $\beta_\t$ and identifying edges of $\beta_\t$ according to this pairing, we obtain a new partition $\widetilde P$ of the quotient space $A'$.  If $m(\t)$ separates the boundary components of $A$, then $A'$ is a planar disc.  If $m(\t)$ doesn't separate the boundary components of $A$, then $A'$ is an annulus.  Then $\Psi''$ induces a map $\widetilde\Psi\co\widetilde P^{(2)} \to \Delta'$ with the desired properties, see Figure \ref{fig3}.

\end{proof}


\begin{defn}

Recall that $G$ has a presentation
$$\pres{ A\cup\{t_i\} }{\{u^{t_i}_{i,s} = v_{i,s} \} \text{ for } \ i = 1, . . . , k \text{ and } s = 1, . . . , j_i }$$
where $U_i = \langle u_{i,1}, . . . , u_{i,j_i}\rangle, V_i = \langle v_{i,1}, . . . , v_{i,j_i}\rangle$ are free subgroups with free generating sets $\{u_{i,j}\}$, $\{v_{i,j}\}$ respectively and $t_i$ are stable letters.

Let $X_i$ be the midpoints of the set of edges $\{(g,t_i) \ | \ g\in U_i\}$ in $\Gamma(G,S)$ .

By Britton's lemma, $gX_i$ separates $\Gamma(G,S)$ for every $g\in G$.  Let $x_1, x_2$ be two points in $X_i$ such that $x_2 =_{G} x_1u_{i,j}$.  Then in $\Gamma^2(G,S)$ we can find an arc joining $x_1$ to $x_2$ which intersects $\Gamma(G,S)$ only at $x_1$ and $x_2$. Let $T_i$ be the subset of $\Gamma^2(G,S)$ obtained by connecting all such points of $X_i$ by arcs which intersect $\Gamma(G,S)$ only at their endpoints.  Since $U_i$ is free, $T_i$ is a tree.  Then $T_i$ separates $\Gamma^2(G,S)$ and will be called the \emph{median tree} for $X_i$.  Notice that $X_i$, $T_i$ are not cellular subset of $\Gamma(G,S)$ or $\Gamma^2(G,S)$, even thought they do have a natural cellular structure.

Let $Z, Z'$ be subsets of $\Gamma(G,S)$.  We will say that $Z,Z'$ are \emph{$t$-separated} if there exists $g\in G$ and $i$ such that $Z, Z'$ are in distinct components of $\Gamma(G,S)\backslash gX_i$.  This is equivalent to saying that as subsets of  $\Gamma^2(G,S)$; $Z, Z'$ are in distinct components of $\Gamma^2(G,S)\backslash gT_i$.

\end{defn}

\begin{rmk}
Notice that $t$-separated does not imply $U_i$-separated or $V_i$-separated. Let $Z$ the set of vertices of $\Gamma(G,S)$ that have a label without pinches which begins with the letter $t_1$.  Let $Z'$ be the remainder of the vertices of $\Gamma(G,S)$.  Then $Z,Z'$ are in distinct components of $\Gamma(G,S)\backslash X_1$.  Since $Z\cup Z'$ contains all the vertices of $G$, they cannot be $U_i$-separated or $V_i$ separated for any $i$.  The point is that $X_i$ separates by removing midpoints of edges and $gV_i$ or $gU_i$ separates by removing vertices.
\end{rmk}

\begin{lem}\label{translates2}\ulabel{translates2}{Lemma}
Suppose $U_i$ is proper, $V_i$ is proper, or the number of stable letter in $S$ is greater than 1 .  Let $\gamma$ be a loop in $\Gamma(G, S)$ and $N> \diam {\gamma}$.  Then there exists elements $\{g_1,\cdots, g_N\}$ in $G$ such that $g_i\cdot\gamma, g_j\cdot\gamma$ are $t$-separated and $|g_ig_j^{-1}|\geq 2N$ for all $i\neq j$; and $|g_i|\leq 4N$.
\end{lem}

\begin{proof}If $U_i$ or $V_i$ is proper, then $\{ g_j\}$ can be constructed as in \ref{etranslates}. If $S$ has at least two stable letters, then let $g_i = t_1^N t_2^it_1^{-N}$.  In any of the three cases, the proof of \ref{etranslates} also shows that the loops $\{g_j\cdot\gamma\}$ are pairwise $t$-separated.
\end{proof}

\begin{lem}\label{medians}\ulabel{medians}{Lemma}
Suppose that $\theta\co \Delta^{(2)} \to \Gamma^2(G,S)$ is the canonical label preserving cellular map from a van Kampen diagram $\Delta$ over $\pres{S}{R}$ to the Cayley complex.  Then $\theta^{-1}(gT_i)$ is a set of medians of $t_i$-bands in $\Delta$.
\end{lem}

\begin{proof}
The only cells in $\Gamma^2(G,S)$ intersecting $gT_i$ are those corresponding to relations of the form $u_{i,j}^t=v_{i,j}$.  The preimage of each edge of $gT_i$ is a median of such a cell in $\Delta$.
\end{proof}

\begin{lem}\label{medians2}\ulabel{medians2}{Lemma}
Suppose that $\Delta_A$ is an annular diagram such that the components of $\theta(\partial \Delta_A)$ are $t$-separated. Then there exist a $t$-annulus in $\Delta_A$ which separates the boundary components of $\Delta_A$.
\end{lem}

\begin{proof}
Since the components of $\theta(\partial \Delta_A)$ are $t$-separated, there exists $g\in G$ and $i$ such that they are in distinct components of $\Gamma^2(G,S)\backslash gT_i$.  Then $\theta^{-1}(gT_i)$ separates the components of $\partial \Delta_A$ and the result follows from \ref{medians}.
\end{proof}

\begin{thm}\label{HNN-free_associated}\ulabel{HNN-free_associated}{Theorem}
Let $G$ be a multiple HNN of a free group with free associated subgroups.  Then either all asymptotic cones of $G$ are simply connected or $G$ has an asymptotic cone with uncountable fundamental group.
\end{thm}

\begin{proof}
If $G$ has only one stable letter and both associated subgroups are not proper, then $G$ has a quadratic Dehn function (see \cite{BrGr}) and every asymptotic cone of $G$ is simply connected.

If there exists an asymptotic cone of $G$ which is not simply connected, then there exists a sequence of loops $\gamma_n$ in $\Gamma(G,S)$ such that $P\bigl(\gamma_n,\frac{|\gamma_n|}{2}\bigr)\geq n$ for all n.  Let $d_n =|\gamma_n|$.  Then $d_n$ diverges \as~ and $\gamma(t) = \bigl(\gamma_n(t)\bigr)$ is a loop which has no finite partition in $\gcon$.

Using \ref{translates2}, we can choose $S_n = \{g_{n,1},\cdots, g_{n, k_n}\}$ of element of $G$ such that

\begin{enumerate}[a)]
    \item if $i\neq j$, then $g_{n, i}\cdot\gamma_n$ and $g_{n,j}\cdot\gamma_n$ are $t$-separated and
    \item for all $i$, $2 \diam{\gamma} d_n\leq |g_{n,i}| \leq 4 \diam{\gamma}d_n$.
\end{enumerate}

\begin{clm*}
Let $g= (g_n), h= (h_n)$ be distinct elements in $\prod^\omega S_n$.  Then $g\cdot\gamma$ is a well-defined loop $\gcon$ and $g\cdot\gamma$ is not homotopic to $h\cdot\gamma$.
\end{clm*}

The first assertion follows from the fact that $g_n$ grows big O of the scaling sequence.

Suppose that $g\cdot\gamma$ is homotopic to $h\cdot\gamma$. Then we have a homotopy $h\co A\to \gcon$ between the two loops where $A$ is a planar annulus.  Let $P$ be a partition of $A$ where each piece is a triangle such that $\diam {h(D)} \leq \frac{1}{84L}$ for each piece $D$ of $P$.  Then we can chose partitions $\Pi_n\co P^{(0)} \to \Gamma(G,S)$ such that $\bigl(\Pi_n(x)\bigr) = h(x)$ for all $x\in P^{(0)}$.  As in Remark \ref{induce}, $\Pi_n$ induces a $t$-shortest partition $\Psi_n\co P^{(2)}\to \Delta_n'$ where $\Delta_n'$ is an annular van Kampen diagram where both boundary paths are labeled by $\lab{\gamma_n}$.  The $\mesh {(\Psi_n)}\leq \frac{|\gamma_n|}{60L} + o(|\gamma_n|)< \frac{|\gamma_n|}{30L}$ \as.  \ref{diameter2} implies that the $\diam{\theta\comp\Psi(D)}\leq 5L\mesh{(\Psi)}< \frac{|\gamma_n|}{6}$ \as.

Since $g\neq h$, $g_n\neq h_n$ \as~ and  the loops $g_n\cdot\gamma_n$ and $h_n\cdot\gamma_n$ are $t$-separated \as.  \ref{medians2} implies that there exists a $t$-annulus in $\Delta_n$ which separates the two boundary components of $\Delta_n$ \as.  \ref{removing} implies we can remove this $t$-annulus to obtain a partition $\widetilde\Psi_n$ of a circular diagram $\Delta_n'$ with $\lab{\partial\Delta_n'}= \lab{\gamma_n}$ \as.  Notice that $\mesh{(\widetilde\Psi_n)}< 3(\frac{|\gamma_n|}{6} + K)$ and has at most $9F^2 + 4F$ where $F$ is the number of pieces of $P$.  This then contradicts our choice of $\gamma_n$.

\end{proof}

\newpage

\newpage
\bibliographystyle{plain}
\bibliography{bib}

\end{document}